\documentclass[12pt]{article}
\usepackage{graphicx}
\usepackage{epstopdf} 
\usepackage{amsfonts}
\usepackage[english]{babel}
\usepackage{amsmath}
\usepackage{bm}
\usepackage{hyperref}
\usepackage{enumerate}
\usepackage{fullpage}

\usepackage{floatrow} 
\newfloatcommand{capbtabbox}{table}[][\FBwidth]

\usepackage{natbib} 
\setcitestyle{authoryear,open={(},close={)}}
\DeclareRobustCommand{\VAN}[3]{#2}

\newtheorem{theorem}{Theorem}

\newtheorem{lemma}[theorem]{Lemma}
\newtheorem{proposition}[theorem]{Proposition}

\newenvironment{proof}[1][Proof]{\par\noindent\textbf{#1.} }{\ \rule{0.5em}{0.5em}\par}
\makeindex

\newcommand{\smbp}{SmBP}
\newcommand{\extra}{\mathcal{R}}

\newcounter{HPcont} 
\renewcommand{\theHPcont}{{(A-\arabic{HPcont})}}
\newenvironment{assume}{
	\begin{itemize}\refstepcounter{HPcont} \item[\!\!\!\!\!\!\!\!\!\!\textbf{\theHPcont}]} {\end{itemize}\par}

\newcounter{HPcontALDO}   
\renewcommand{\theHPcontALDO}{{(B-\arabic{HPcontALDO})}}
\newenvironment{assumeALDO}{
	\begin{itemize}\refstepcounter{HPcontALDO} \item[\!\!\!\!\!\!\!\!\!\!\textbf{\theHPcontALDO}]} {\end{itemize}\par}

\newcounter{Discussion}
\renewcommand{\theDiscussion}{{D.\arabic{Discussion}}}
\newcommand{\discussion}{
	\refstepcounter{Discussion}\noindent\textbf{\theDiscussion{ }--{ }}
}

\begin{document}

\title{Some Insights About the Small Ball Probability Factorization for Hilbert Random Elements}
\author{Enea~G.~Bongiorno, Aldo~Goia \\
Universit\`{a} del Piemonte Orientale, \\
enea.bongiorno@uniupo.it, aldo.goia@uniupo.it }
\maketitle

\abstract{
Asymptotic factorizations for the small--ball probability (\smbp{}) of a Hilbert valued random element $X$ are rigorously established and discussed. 
In particular, given the first $d$ principal components (PCs) and as the radius $\varepsilon$ of the ball tends to zero, the \smbp{} is asymptotically proportional to (a) the joint density of the first $d$ PCs, (b) the volume of the $d$--dimensional ball with radius $\varepsilon$, and (c) a correction factor weighting the use of a truncated version of the process expansion. 
Moreover, under suitable assumptions on the spectrum of the covariance operator of $X$ and as $d$ diverges to infinity when $\varepsilon$ vanishes, some simplifications occur. In particular, the \smbp{} factorizes asymptotically as the product of the joint density of the first $d$ PCs and a pure volume parameter.
All the provided factorizations allow to define a \textit{surrogate intensity} of the \smbp{} that, in some cases, leads to a genuine \textit{intensity}.
To operationalize the stated results, a non--parametric estimator for the \textit{surrogate intensity} is introduced and it is proved that the use of estimated PCs, instead of the true ones, does not affect the rate of convergence. Finally,  as an illustration, simulations in controlled frameworks are provided.
}



\section*{Introduction}

For a random element $X$ valued in a general metric space, the measure of how it concentrates over such space plays a central role in statistical analysis. If $X$ is a real random vector, its joint density is, in a natural way, that measure.
In fact, in practical situations, the density is helpful in defining mixture models, in detecting latent structure, in discriminant analysis, in robust statistics to identify outliers and so on.
When observed data are curves, surfaces, images, objects or, briefly, \emph{functional data} (see e.g.~monographs \citealp{fer:vie06, hor:kok12, ram:sil05}, and \citealp{bon:goi:sal:vie14} for recent contributions), the dimensionality of the space to which the data belong raises problems in defining an object that plays the same role of the joint density distribution in the multivariate context. The main problem is that, without an underlying dominant probability measure, the Radon--Nikodym derivative can not be straightforwardly applied.
To manage this, a concept of ``surrogate density'' is derived from the notion of small--ball probability (\smbp{} in the sequel) of a random element $X$, that is
\begin{equation*}
\varphi \left( x,\varepsilon \right) =\mathbb{P}\left( \Delta(X,x) <\varepsilon \right) ,
\end{equation*}%
where $x$ is in the same space where $X$ takes its values, $\varepsilon$ is a positive real number and $\Delta$ is a semi--metric. 
The behaviour, as $\varepsilon $ vanishes, of $\varphi \left( x,\varepsilon \right)$ provides information about the way in which $X$ concentrates at $x$. Such feature stimulates the study of the \smbp{} in different settings: the limiting behaviour has been developed from a theoretical point of view (for instance, refer to the small tails/deviations theory \citealp{li:sha01,lif12} and references therein), in functional statistics the \smbp{} was used to derive asymptotics in mode estimations (see, e.g.\@ \citealp{dab:fer:vie07,del:hal10,fer:kud:vie12,gas:hal:pre98}), as well as in non--parametric regression literature in evaluating the rate of convergence of estimators (see, e.g.\@ \citealp{fer:vie06,fer:mas:vie07}). 
\\%
Often, the necessity to have available a surrogate density for $X$ has brought to assume (as done, for instance, in \citealp{fer:kud:vie12, gas:hal:pre98}) that
\begin{equation}
\varphi (x,\varepsilon )=\Psi \left( x\right) \phi \left( \varepsilon
\right) +o\left( \phi \left( \varepsilon \right) \right) ,\qquad \varepsilon
\rightarrow 0,  \label{eq:SmBP_factorization}
\end{equation}%
where  $\Psi $, depending only on the center $x$, plays the role of the \emph{surrogate density} of the random element $X$, whilst  $\phi \left( \varepsilon \right) $ is a kind of ``volume parameter'' which does not depend on the spatial term.
It is worth noticing that $\Psi$ is also the \emph{intensity} of the \smbp{}.
\\%
Although to break the dependence on $x$ and $\varepsilon$ supplies a clear modelling advantage and the existence of $\Psi\left(x\right)$ is desirable (especially from a statistical perspective), factorization \eqref{eq:SmBP_factorization} can be derived only in particular settings. Notable examples are the case of Gaussian processes (e.g.\@ \citealp{li:sha01, lif12} and references therein) and the one of fractal processes for suitable semi--norms $\Delta$ (e.g.\@ \citealp[Chapter~13]{fer:vie06}).
Hence, a crucial task is to study some asymptotic factorizations of the \smbp{} that lead to a definition of its \textit{intensity} or, at least, when it is not possible to completely isolate the dependence on $x$ and $\varepsilon$, a \textit{surrogate intensity}. 
\\ %
In what follows, we assume $X$ to be a random element in an Hilbert space with $\Delta$ being the metric induced by the Hilbert norm and, without loss of generality, we deal with random curves on the space of square integrable functions on the unit interval. A first factorization of the \smbp{} that allows to define a \textit{surrogate intensity} was provided by \cite{del:hal10}: besides some technical hypothesis, on the spectrum of the covariance operator of $X$ and assuming that principal components of $X$ are independent with positive and sufficiently smooth marginal density functions $\{\tilde{f}_{j}\}$, the authors showed that
\begin{equation}
\varphi (x,\varepsilon )\sim \prod_{j\leq d}\tilde{f}_{j}\left( x_{j}\right)
\phi(\varepsilon,d),\qquad \varepsilon \rightarrow 0,  \label{eq:DH_factorization}
\end{equation}%
where $x_{j}$ is the projection of $x$ over the $j$-th principal axis, $\phi(\varepsilon,d)$ is volumetric term and $d=d(\varepsilon )$ is the number of considered terms of the decomposition
diverging to infinity as $\varepsilon$ vanishes. Now, from the application point of view, the independence assumption appears quite restrictive and the spatial factor $\prod_{j\leq d}\tilde{f}_{j}$ results just a \textit{surrogate intensity} of the \smbp{} because of the dependence between $d$ and $\varepsilon$.
Moreover, one may wonder if the principal component analysis is necessary to obtain the factorization.

The first part of the present work is devoted to propose more general factorizations for the SmBP that relax the hypothesis of independence, and to identify situations in which it is possible to obtain a genuine \textit{intensity}. 
The first stated factorization holds for any positive integer $d$:
\begin{equation*}
	\varphi (x,\varepsilon )\sim f_d(x_{1},\ldots ,x_{d}) V_d(\varepsilon) \extra\left(x,\varepsilon ,d\right) ,\qquad \text{\textnormal{as }}\varepsilon \to 0,  
\end{equation*}
where $f_d$ is the joint distribution of the first $d$ principal components, $V_d(\varepsilon)$ is the volume of a $d$--dimensional ball with radius $\varepsilon$ and $\extra\left(x,\varepsilon ,d\right)\in(0,1]$ denotes an extra factor compensating the use of $(x_{1},\ldots ,x_{d})$ instead of $x$. Such factorization benefits from the fact that $d$ is fixed but, in general, a genuine \textit{intensity} can not be defined because $\extra$ depends on both $x$ and $\varepsilon$. Such dependency is bypassed and hence an \textit{intensity} obtained if one introduces suitable assumptions on the probability law of the process and/or on the point $x$ at which the factorization is evaluated.
\\%
Moving from this first factorization, we prove that
\begin{equation*}
\varphi (x,\varepsilon )\sim f_d(x_{1},\ldots ,x_{d}) \phi(\varepsilon, d) ,\qquad \text{\textnormal{as }}\varepsilon \to 0, \textnormal{ and }d(\varepsilon)\to\infty,  
\end{equation*}
where $\phi(\varepsilon,d)$ is a volume parameter that depends on the decay rate of $\{\lambda_j\}$, the eigenvalues of the covariance operator of $X$.
Such result canditates, in a very natural way, the joint density distribution $f_d$ to be a \textit{surrogate intensity} of the \smbp{} and, under suitable assumptions on $X$, allows to define an \textit{intensity}.
Furthermore, it turns out that, the first factorization is basis free, while for the second one the principal components basis is optimal in some sense.

In the second part of the paper, in order to make available the \textit{surrogate intensity} of the \smbp{} for statistical purpose, we propose a multivariate kernel density approach to estimate $f_d$. Under general conditions, we prove that, although the estimation procedure involves the estimated principal components instead of the true ones, the estimator achieves the classical non--parametric rate of convergence.
To show how such estimator performs on finite sample frameworks, we study its behaviour by means of simulated processes whose \textit{intensity} is known.

The paper outline goes as follow: Section~\ref{sec:notations_and_assumptions} introduces the framework, Section~\ref{sec:SmBP_approx_d_fixed} considers the factorization of the \smbp{} when $d$ is fixed whereas Section~\ref{sec:SmBP_approx_moving_d_to_infinity} when $d$ diverges to infinity as $\varepsilon$ vanishes. Section~\ref{sec:joint_dist_estimate} provides the statistical asymptotic theorem in estimating the joint density $f_d$. Section~\ref{sec:applications} illustrates some numerical examples. Finally Section~\ref{sec:proofs} collects all the proofs.

\section{Preliminaries}
\label{sec:notations_and_assumptions}

Let $\left( \Omega ,\mathcal{F},\mathbb{P}\right) $ be a probability space and $\mathcal{L}_{\left[ 0,1\right] }^{2}$ be the Hilbert space of square integrable real functions on $\left[ 0,1\right] $ endowed with the inner product $\left\langle g,h\right\rangle =\int_{0}^{1}g\left( t\right) h\left( t\right) dt$ and the induced norm $\left\Vert g\right\Vert ^{2}=\left\langle g,g\right\rangle $. Consider a measurable map $X$ defined on $(\Omega ,\mathcal{F})$ taking values in $(\mathcal{L}_{\left[ 0,1\right]}^{2},\mathcal{B})$, where $\mathcal{B}$ denotes the Borel sigma--algebra induced by $\Vert \cdot \Vert $. Denote by
\begin{equation*}
\mu _{X}=\left\{ \mathbb{E}\left[ X\left( t\right) \right] ,t\in \left[ 0,1%
\right] \right\} ,\qquad \text{and}\qquad \Sigma \left[ \cdot \right] =%
\mathbb{E}\left[ \left\langle X-\mu _{X},\cdot \right\rangle \left( X-\mu
_{X}\right) \right],
\end{equation*}%
the mean function and covariance operator of $X$ respectively. Let us consider the Karhunen--Lo\`{e}ve expansion associated to $X$ (see e.g. \citealp{bos00}): denoting by $\left\{ \lambda _{j},\xi _{j}\right\} _{j=1}^{\infty }$ the decreasing to zero sequence of non--negative eigenvalues and their associated orthonormal eigenfunctions of the covariance operator $\Sigma $, the random curve $X$ admits the following representation
\begin{equation*}
X\left( t\right) =\mu _{X}\left( t\right) +\sum_{j=1}^{\infty }\theta
_{j}\xi _{j}\left( t\right) ,\qquad 0\leq t\leq 1,  
\end{equation*}%
where $\theta _{j}=\left\langle X-\mu _{X},\xi _{j}\right\rangle $ are the
so--called principal components (PCs in the sequel) of $X$ satisfying
\begin{equation*}
\mathbb{E}\left[ \theta _{j}\right] =0,\qquad Var\left( \theta _{j}\right)
=\lambda _{j},\qquad \mathbb{E}\left[ \theta _{j}\theta _{j^{\prime }}\right]
=0,\qquad j\neq j^{\prime }.
\end{equation*}%
It is just the case to recall that $\{\xi _{j}\}_{j=1}^{\infty }$ provides an orthonormal basis of the considered Hilbert space and that the Karhunen--Lo\`{e}ve expansion, taking advantage of the euclidean underling structure, isolates the manner in which the random function $X(\omega ,t)$ depends upon $t$ and upon $\omega $. 
\\
In order to achieve our aims, let us consider the following assumptions.
\begin{assume}\label{ass:zero_mean}
	The process is centered, that is $\mu _{X}=0$.
\end{assume}
\begin{assume}\label{ass:boundedness_of_x}
	The center of the ball $x\in \mathcal{L}_{\left[ 0,1\right]}^{2}$ is sufficiently close to the process in its high--frequency part, that is
	\begin{equation}  \label{eq:E[W^2]_bounded}
	x_j^2 \le C_1 \lambda_j, \qquad \textnormal{for any } j\ge 1
	\end{equation}
	where $x_{j}=\left\langle x,\xi _{j}\right\rangle $ for some positive constant $C_1$.
\end{assume}	

\noindent The latter is not a restrictive condition since it holds whenever $x$ belongs to the reproducing kernel Hilbert space generated by the process $X$:
\begin{equation}  \label{eq:def_RKHS}
RKHS(X)=\left\{x\in \mathcal{L}_{\left[ 0,1\right] }^{2} : \sum_{j\ge 1}
\lambda_j^{-1} \langle x,\xi_j\rangle^2 < \infty \right\}.
\end{equation}
Roughly speaking, $x$ is an element of $ RKHS(X)$ only if it is ``at least smooth as the covariance function'', see \citealp[p.\@ 13 and p.\@ 69]{ber:tho04}. Note that $RKHS(X)$ is a very large subspace of $H$ including the finite dimensional ones; in fact, if $x_j=0$ for any $j\ge d$ and any $d\in\mathbb{N}$, then $x\in RKHS(X)$. Furthermore, \ref{ass:boundedness_of_x} is not unusual since it is equivalent to $\sup_{j\geq 1}\mathbb{E}\left[ (\theta_{j} - x_j)^{2} /\lambda_j\right] <\infty$ that was used, for similar purpose by \citealp[Condition~(4.1)]{del:hal10}. 
\begin{assume}\label{ass:boundedness_second_derivatives_of_f}
	Denote by $\Pi _{d}$ the projector onto the $d$--dimensional space spanned by $\{\xi _{j}\}_{j=1}^{d}$. The first $d$ PCs, namely $\boldsymbol{\theta }=\Pi _{d}X=(\theta _{1},\ldots ,\theta _{d})^{\prime }$, admit a joint strictly positive probability density, namely $\boldsymbol{\vartheta }\in \mathbb{R}^{d}\mapsto f_{d}(\boldsymbol{\vartheta })$.
	Moreover, $f_{d}$ is twice differentiable at $\boldsymbol{\vartheta }=\left( \vartheta _{1},\dots ,\vartheta _{d}\right) ^{\prime }\in \mathbb{R}^{d}$ and there exists a positive constant $C_2 $ (not depending on $d$) for which
	\begin{equation}	\label{eq:boundedness_second_derivatives_of_f}
		\left\vert \frac{\partial ^{2}f_{d}}{\partial \vartheta _{i}\partial \vartheta _{j}}(\boldsymbol{\vartheta })\right\vert \leq \frac{C_2 }{\sqrt{\lambda _{i}\lambda_{j}}} f_{d}(x_1,\ldots, x_d),
	\end{equation}
	for any $d\in\mathbb{N}$, $i,j \in\{1,\ldots, d\}$ and $\boldsymbol{\vartheta }\in D^x=\left\{ \boldsymbol{\vartheta }\in \mathbb{R}^{d}:\sum_{j\leq d}\left( \vartheta _{j}-x_{j}\right) ^{2}\leq \rho ^{2}\right\} $ for some $\rho \geq \varepsilon $.		
\end{assume}
\noindent From now on, with a slight abuse of notation and when it is clear from the context, $f_{d}(x)$ denotes $f_{d}\left( x_{1},\dots ,x_{d}\right) $.
\\
To better appreciate the meaning of \eqref{eq:boundedness_second_derivatives_of_f}, note that it can be derived in an intuitive way considering $\mathbf{W}^x=\left( W_{1}, \dots ,W_{d}\right) ^{\prime }$, the deterministic translation of the component--wise standardized version of the PCs defined by
\begin{equation*}
	W^x_{j}=\frac{1}{\sqrt{\lambda_{j}}}\left\langle X-x,\xi _{j}\right\rangle =
	\frac{\theta_j- \left\langle x,\xi _{j}\right\rangle}{\sqrt{\lambda_{j}}}.
\end{equation*}
In fact, \eqref{eq:boundedness_second_derivatives_of_f} is equivalent in assuming the boundedness of the second derivative of the density probability function $g_{d}^x$ of the random vector $\mathbf{W}^x$. Since the latter is a linear transformation of $\boldsymbol{\theta%
}$,  condition \eqref{eq:boundedness_second_derivatives_of_f} is equivalent to
\begin{equation*}  
{\left\vert \frac{\partial^{2}g_d^x}{\partial w_{i}\partial w_{j}}(\mathbf{w})\right\vert}  \leq C_2  g_d^x(\mathbf{0}),
\end{equation*}
for any $d\in\mathbb{N}$, $i,j \in\{1,\ldots, d\}$ and $\mathbf{w}\in D^{\prime}=\left\{ \mathbf{w}\in \mathbb{R}^{d}:\sum_{j\leq 	d}w_j^{2}\lambda_j \leq \rho^2 \right\} $ for some $\rho\ge \varepsilon$.
It is worth noting that \ref{ass:boundedness_second_derivatives_of_f} is not restrictive: it includes, for instance, the case of Gaussian Hilbert valued processes. 

\section{Approximation results for a given $d$}
\label{sec:SmBP_approx_d_fixed}

To state the main result of this section, let us consider a finite positive integer $d$, a given point $x\in \mathcal{L}_{\left[ 0,1\right] }^{2}$ and define 
\begin{equation}  \label{eq:def_of_S_and_R}
	S=S(x,\varepsilon ,d)  =\frac{1}{\varepsilon ^{2}}\sum_{j\geq d+1}\left(	\theta _{j}-x_{j}\right) ^{2},
	\qquad 
	\extra\left(x,\varepsilon ,d\right) =\mathbb{E}\left[ \left( 1-S\right)^{d/2}\mathbb{I}_{\left\{ S< 1\right\} }	\right],
\end{equation}
and
\begin{equation*}
	V_d(\varepsilon) = \frac{\varepsilon ^{d}\pi^{d/2}}{\Gamma \left( d/2+1\right) },
\end{equation*}
that is the volume of the $d$--dimensional ball with radius $\varepsilon$. 
\begin{theorem}
	\label{prop:SB} Let $X$ be a process as above, $\varphi \left( x,\varepsilon \right)$ be the small ball probabilities of $X$, assume \ref{ass:zero_mean}, \ldots, \ref{ass:boundedness_second_derivatives_of_f} and define
	\begin{equation}\label{eq:SmBP_approximation}
	\varphi _{d}(x,\varepsilon ) = f_d(x) V_d(\varepsilon) \extra\left(x,\varepsilon ,d\right)
	,\qquad \text{\textnormal{for }}\varepsilon >0.
	\end{equation}%
	Then
	\begin{equation}
	\left\vert \varphi (x,\varepsilon )-\varphi _{d}(x,\varepsilon )\right\vert
	\leq C_2  \frac{\varepsilon^{2}}{2\lambda_d} \varphi_d(x,\varepsilon),\qquad \text{\textnormal{for }}%
	\varepsilon >0  \label{eq:SmBP_inequality_r_fixed}
	\end{equation}%
	that is
	\begin{equation}
	\varphi (x,\varepsilon )\sim f_d(x) V_d(\varepsilon) \extra\left(x,\varepsilon ,d\right) ,\qquad \text{\textnormal{for }}\varepsilon
	\rightarrow 0.  \label{eq:SB_rough_version}
	\end{equation}
\end{theorem}

In other words, for a fixed positive integer $d $ and as $\varepsilon\rightarrow 0$, above theorem states that the \smbp{} $\varphi (x, \varepsilon )$ behaves as $\varphi_d (x, \varepsilon)$ (the usual first order approximation of the \smbp{} in a $d$--dimensional space $f_d(x)V_d(\varepsilon)$) up to the scale factor $\extra\left(x,\varepsilon ,d\right)$. The latter, depending on $x$ only through its high--frequency components $\{x_j\}_{j\ge d+1}$, can be interpreted as a corrective factor compensating the use of a truncated version of the process expansion. 
Note that changing $d$ affects all the terms in the factorization but not the asymptotic \eqref{eq:SB_rough_version}.

%
Because of $\extra(x,\varepsilon,d)$, the dependence on $x$ and $\varepsilon$ can not be isolated in \eqref{eq:SB_rough_version} and hence an \textit{intensity} of the \smbp{} is not, in general, available.
On the other hand, there exist some specific situations in which a genuine \textit{intensity} can be defined from above factorization. We have identified three of these, namely: 
\begin{enumerate}[a)]
	\item $\extra(x,\varepsilon,d)$ is independent on $x$, 
	
	\item
	there exists a finite positive integer $d_0$ such that, for any $d\ge d_0$, $\extra(x,\varepsilon,d)=1$, 
	
	\item 
	for any $x$, $\displaystyle \left\{
\begin{array}{l}
\extra(x,\varepsilon,d) \to 1,
\\
\varphi(x,\varepsilon)\sim f_d(x)V_d(\varepsilon),
\end{array}
\right.\qquad \varepsilon \to 0, \ d(\varepsilon)\to \infty.
$
\end{enumerate}

In the following, we discuss points a) and b), whereas point c) is discussed in Section~\ref{sec:SmBP_approx_moving_d_to_infinity} since requires additional arguments.
In the last discussion we spend few words about the consequences of choosing a basis different from the Karhunen--Lo\`{e}ve one.

\discussion\label{rem:wiener_case} %
\textbf{$\extra(x,\varepsilon,d)$ is independent on $x$.} 
Consider, for instance, $x_j=0$ for any $j\ge d_0+1$ (i.e.\@ $x$ belongs to the space spanned by $\{\xi_1,\ldots, \xi_{d_0}\}$). Then applying Theorem~\ref{prop:SB}, for any $d\ge d_0$, we have
\begin{equation*}
\varphi(x,\varepsilon)\sim f_d(x) V_d(\varepsilon) \extra\left(\varepsilon ,d\right)
,\qquad \text{\textnormal{as }}\varepsilon \to 0,
\end{equation*}
where $V_d(\varepsilon) \extra\left(\varepsilon ,d\right)$ represents a pure volumetric term while $f_d$ is an \textit{intensity} of the \smbp{} evaluated at $x$.
Let us consider the remarkable Gaussian process for which
\begin{equation*}
\varphi (x,\varepsilon )\sim \exp \left\{ -\frac{1}{2}\sum_{j\leq d}\frac{x_{j}^{2}}{\lambda _{j}}\right\}\cdot 
\frac{V_d(\varepsilon) \extra(\varepsilon, d)}{\prod_{j\leq d}\sqrt{2\pi \lambda _{j}}} 
= \Psi_d(x) \cdot \mathcal{V}_d(\varepsilon)
,\qquad \text{as }\varepsilon \longrightarrow 0,
\end{equation*}
where, for any $d\ge d_0$,  $\Psi_d(x) = \Psi_{d_0}(x) = \exp \left\{ -\sum_{j\leq d_0}x_{j}^{2}/(2\lambda _{j})\right\} $ is the \textit{intensity} of the \smbp{} evaluated at $x$. 
If we further specialize the above case to the Wiener process on $[0,1]$, we can show that $\Psi_{d_0}(x)$ is coherent with already known results (see, for instance, \citealp[Theorem~3.1]{li:sha01} and \citealp[Example~5.1]{der:feh:mat:sch03}). In fact, the Karhunen--Lo\`{e}ve decomposition of a Wiener process is known to be
\begin{equation*}
W\left( t\right) =\sum_{j=1}^{\infty }Z_{j} \xi_j(t) , \qquad\textnormal{with}\qquad
\left\{
\begin{array}{l}
t\in [0,1],\\
Z_j \sim N(0,1)\ \ i.i.d.,
\\
\xi _{j}\left( t\right) =\sqrt{2} \sin\left(\left( j-0.5\right) \pi t\right)/\sqrt{\lambda_j},
\\
\lambda _{j}=\left( j-0.5\right)^{-2}\pi ^{-2},
\end{array}
\right.
\end{equation*}
and it is known that
\begin{equation*}
\varphi (x,\varepsilon )\sim \exp \left\{ -\frac{1}{2}\int_0^1 x^{\prime }\left(
t\right) ^{2}dt\right\} \frac{4\varepsilon}{\sqrt{\pi} \exp \left\{ \frac{1}{8\varepsilon ^{2}}\right\}}, \qquad \varepsilon\rightarrow 0,
\end{equation*}
where $x(t)$ is sufficiently smooth. Thus, the latter must be equivalent to $\Psi_{d_0}(x) \cdot \mathcal{V}_d(\varepsilon)$ as $\varepsilon$ goes to zero and for any $d\ge d_0$. Since we are interested in the definition of an intensity, we do not care about the volumetric part and we just compare the spatial parts. Given $ x\left( t\right) =\sum_{j=1}^{d_0}b_{j}\xi_{j}\left(t\right)$ where $b_{j}\in \mathbb{R}$, then $x_{j}$ equals $b_{j}$, for $j=1,\dots ,d_0$ and is null otherwise. Moreover, straightforward computations lead to 
\begin{equation*}
\exp \left\{ -\frac{1}{2}\int_{0}^{1}x^{\prime }\left( t\right)^{2}dt\right\} = \exp \left\{ -\frac{1}{2}\sum_{j=1}^{d_0}b_{j}^{2}\right\} = \Psi_{d_0}(x).
\end{equation*}%
\\
The special case of Wiener process is exploited in Section~\ref{sec:applications_MB} to construct numerical examples in estimating the \textit{intensity} $\Psi_{d_0}(x)$.

\discussion\textbf{The case $\extra(x,\varepsilon,d)=1$.}\label{rem:extra_equals_1} Consider $X$ a $d_0$-dimensional process (that is when it takes values in a $d_0$--dimensional subspace of the Hilbert space). Then $\lambda_j = 0$ for any $j\ge d_0+1$, \ref{ass:boundedness_of_x} leads to $x_j=\theta_j=0$ and  $\extra(x,\varepsilon,j)=1$, for any $j\ge d_0+1$.
Moreover, Theorem~\ref{prop:SB} can be applied only for $d\le d_0$ because $f_{d_0+1}$ is not strictly positive and hence \ref{ass:boundedness_second_derivatives_of_f} fails. Consequently, $\varphi(x,\varepsilon) \sim f_{d_0}(x)V_{d_0}(\varepsilon)$ that is the usual first order approximation of the $d_0$--dimensional process and $f_{d_0}$ is the \textit{intensity} of the \smbp{} of the process.

\discussion\textbf{Changing the basis.} \label{rem:changing_basis}
Let $\{\xi_{j}\}_{j=1}^{\infty }$ be an orthonormal basis  of the Hilbert space arranged so that the sequence $Var(\left\langle X,\xi_{j}\right\rangle)=\lambda _{j}$ is sorted in descending order. Then Theorem~\ref{prop:SB} still holds.

\section{Approximation results when $d$ depends on $\varepsilon$}
\label{sec:SmBP_approx_moving_d_to_infinity}

The goal of this section is to establish which conditions on $X$ allow to simplify \eqref{eq:SB_rough_version}, to get
\begin{equation*}
\varphi (x,\varepsilon )\sim f_{d}(x)V_{d}(\varepsilon ),\qquad \text{%
	\textnormal{as }}\varepsilon \rightarrow 0.
\end{equation*}
In what follows, such result is achieved combining Theorem~\ref{prop:SB} and the limit behaviour of $\extra(x,\varepsilon,d)$ which is strictly related to that of the real random variable $S(x,\varepsilon,d)$ defined by Equation~\eqref{eq:def_of_S_and_R}. On the one hand, whether $d$ and $x$ are fixed, $S$ diverges with $\varepsilon$ tending to zero. On the other hand, if $\varepsilon$ and $x$ are fixed, the larger $d$ the smaller $S$. Hence, one may wonder if it is possible to balance these two effects (as instance, tying the behaviour of $d $ to that of $\varepsilon$) in order to have, for any $x$,
\begin{equation}\label{eq:R_to_1_and_final_factorization}
\displaystyle \left\{
\begin{array}{l}
\extra(x,\varepsilon,d) \to 1,
\\
\varphi(x,\varepsilon)\sim f_d(x)V_d(\varepsilon),
\end{array}
\right.\qquad \varepsilon \to 0,\ d(\varepsilon)\to\infty.
\end{equation}
To do this let us consider the following limit behaviour of $\extra$, as $\varepsilon$ goes to zero and $d$ diverges to infinity.
\begin{proposition}
	\label{pro:E[(1-S)^r/2]_asymptotics} Assume \ref{ass:boundedness_of_x} and suppose that $\sum_{j\geq d+1}\lambda _{j}=o\left( 1/d\right)$, as $d$ goes to infinity.
	Then it is possible to choose $d=d(\varepsilon )$ so that it diverges to infinity as $\varepsilon $ tends to zero and
	\begin{equation}
		d\sum_{j\geq d+1}\lambda _{j}=o(\varepsilon ^{2}).
		\label{eq:r_fixed_for_E(1-S)^r/2_convergence}
	\end{equation}%
	Moreover, as $\varepsilon \rightarrow 0$,
	\begin{equation}
		0\le 1-\extra(x,\varepsilon,d) \leq \frac{C_1(d+2)}{2\varepsilon ^{2}}\sum_{j\geq
			d+1}\lambda _{j}=o(1).  \label{eq:expec_behaviour}
	\end{equation}
\end{proposition}
Let us now consider the following inequality
\begin{equation*}
\left\vert \varphi (x,\varepsilon )-f_{d}\left( x\right) V_{d}(\varepsilon )\right\vert \leq
\left\vert \varphi (x,\varepsilon )-\varphi _{d}(x,\varepsilon )\right\vert
+\left\vert \varphi _{d}(x,\varepsilon )-f_{d}\left( x\right) V_{d}(\varepsilon )\right\vert,
\end{equation*}%
that, thanks to \eqref{eq:SmBP_inequality_r_fixed}, \eqref{eq:expec_behaviour} and to the fact that $0< \extra\le 1$, leads to
\begin{align}
\left\vert \frac{\varphi (x,\varepsilon )}{f_{d}\left( x\right) V_{d}(\varepsilon )}-1\right\vert 
& \leq C_2  \frac{\varepsilon^{2}}{2\lambda_d} \extra(x,\varepsilon,d)
+ \left\vert \extra(x,\varepsilon,d) -1\right\vert,
\nonumber
\\
& \leq C_2  \frac{\varepsilon^{2}}{2\lambda_d} +\frac{C_1(d+2)}{2\varepsilon ^{2}}\sum_{j\geq d+1}\lambda_{j}.
\label{eq:errore_globale}
\end{align}%
%
%
Thus, the wished result holds whenever the right--hand side vanishes as $\varepsilon $ goes to zero, i.e.\@ if there exists $d=d(\varepsilon)$ such that, at the same time the following two conditions hold
\begin{equation} \label{eq:errors_combination}
\varepsilon ^{2}=o\left( \lambda _{d} \right), \qquad \textnormal{and} \qquad (d+2)\sum_{j\geq d+1}\lambda _{j} = o(\varepsilon^{2}).
\end{equation}
In order to obtain \eqref{eq:R_to_1_and_final_factorization}, we combine conditions in \eqref{eq:errors_combination} (plug the first in the second), and we get that eigenvalues must satisfy the \emph{hyper--exponential} decay rate defined by
\begin{equation}\label{eq:rate_eigen_conclusion_1}
\frac{d\sum_{j\geq d+1}\lambda _{j}}{\lambda_d} =o\left( 1\right) ,\qquad \text{\textnormal{as }}%
d\rightarrow \infty.
\end{equation}%
The latter highlights a trade--off between the approximation errors provided by Theorem~\ref{prop:SB} and Proposition~\ref{pro:E[(1-S)^r/2]_asymptotics}. This trade--off is strictly related to a suitable balance between the $d$--th eigenvalue and the terms in the tail of the spectrum of the covariance operator.

It is worth noting that hyper--exponential decay of eigenvalues is a necessary condition to guarantee that the right--hand side of \eqref{eq:errore_globale} vanishes. One may wonder, if it is sufficient as well, that is in other words, if it is possible to define $d=d(\varepsilon)$ so that the errors in \eqref{eq:errors_combination} vanish at the same time as $\varepsilon$ goes to zero. A positive answer to this is furnished by the following.

\begin{theorem}
\label{teo:SmBP_final} Consider hypothesis of Theorem~\ref{prop:SB} and assume that eigenvalues decay hyper--exponentially. It is possible to choose $d=d(\varepsilon)$ so that, if $\varepsilon \rightarrow 0$, then $d\to\infty$ and 
\begin{equation}  \label{eq:SB_final_version}
\varphi (x,\varepsilon ) = f_d\left( x\right) 
V_d(\varepsilon) +o(f_d\left( x\right) V_d(\varepsilon) ).
\end{equation}
\end{theorem}

In what follows, in order to discuss assumptions and consequences of above result, some issues are developed.

\discussion\label{rem:intensity_Gaussian_case_2}\textbf{Again about the intensity of the \smbp{}.}
Because of the relation between $d$ and $\varepsilon$, in general, \eqref{eq:SB_final_version} does not allow to define an \textit{intensity} as commonly intended: anyway, being $f_d$ the only term depending on $x$, it can be considered as a \textit{surrogate intensity}. 
\\
The Gaussian processes or their suitable generalizations, provide examples for which $f_d$ leads to define a genuine \textit{intensity}.
At first, consider a Gaussian process $X$, then for any $x\in \mathcal{L}_{\left[ 0,1\right] }^{2}$ and as $\varepsilon$ goes to zero, $\varphi (x,\varepsilon ) \sim \Psi_{d}(x) \cdot \mathcal{V}_d(\varepsilon)$; see \ref{rem:wiener_case}. 
When $d$ tends to infinity, then $ \Psi_d(x)$ tends to $\Psi (x)=\exp \left\{ -\sum_{j=1}^{\infty } x_{j}^{2}/ (2\lambda _{j})\right\}$ for any $x\in \mathcal{L}_{\left[ 0,1\right] }^{2}$ and $\Psi (x)$ plays the role of the \textit{intensity} of the small--ball probability at $x$. Note that, $\Psi (x)$ is not null if and only if $x$ belongs to $RKHS(X)$, see \eqref{eq:def_RKHS}. In particular, if we consider the Wiener process on $[0,1]$, it can be proven, with arguments analogous to those in \ref{rem:wiener_case} and for any smooth real function $x$ on $[0,1]$, that
	\begin{equation*}
		\Psi (x) = \exp \left\{ -\frac{1}{2}\int_0^1 x^{\prime }\left(
		t\right) ^{2}dt\right\}.
	\end{equation*}
	The latter is coherent with already known results; see, for instance, \citealp[Theorem~3.1]{li:sha01}.
\newline
Another situation in which an \textit{intensity} for the \smbp{} can be also defined occurs whether the PCs are independent each one with density belonging to a subfamily of the exponential power (or generalized normal) distribution (see e.g. \citealp{box:tia73}), that is proportional to $\exp \left\{ -\left( |x_{j}|/\sqrt{\lambda _{j}}\right) ^{q}\right\} $, with $q\geq 2$. In this case, $\Psi $ is given by
\begin{equation*}
\Psi (x)=\exp \left\{ -\frac{1}{2}\sum_{j=1}^{\infty }\left( \frac{|x_{j}|}{%
\sqrt{\lambda _{j}}}\right) ^{q}\right\} ,\qquad \text{\textnormal{for any }}%
x\in \mathcal{L}_{\left[ 0,1\right] }^{2}
\end{equation*}%
and, it is not null if $x$ is in $H(q)=\left\{ x\in \mathcal{L}_{\left[ 0,1\right] }^{2}:\sum_{j=1}^{\infty
}\left( |x_{j}|/\sqrt{\lambda _{j}}\right) ^{q}<\infty \right\}$ that includes $RKHS(X)$ when $q\geq 2$.

\discussion \textbf{An example of hyper--exponential decay.} Consider $\lambda _{j}=\exp \{-\beta j^{\alpha }\}$ with $\beta >0$ and $\alpha >1$.
In this case, for any real number $n\ge 1$, it holds
\begin{equation}
\frac{d\sum_{j\geq d+1}\lambda _{j}}{\lambda_d} \leq \frac{d^n \sum_{j\geq d+1}\lambda _{j}}{\lambda_d} \rightarrow 0,\qquad \text{\textnormal{as }}%
d\rightarrow \infty .  \label{eq:example_super-exp}
\end{equation}%
In fact, some algebra and the Bernoulli inequality give
\begin{align*}
\frac{d^n \sum_{j\geq d+1}\lambda _{j}}{\lambda_d} &
=d^{n}\left( \sum_{j\geq 1}\exp \{\beta d^{\alpha }(1-(1+j/d)^{\alpha
})\}\right) \\
& \leq d^{n}\left( \sum_{j\geq 1}\exp \{-\beta \alpha d^{\alpha
	-1}j\}\right) .
\end{align*}%
Since $\exp \{-\beta \alpha d^{\alpha -1}j\}\leq (j^{2}d^{n+\delta })^{-1}$ eventually (with respect to $d$) holds for some positive $\delta $ and for each $j\in \mathbb{N}$, \eqref{eq:example_super-exp} is obtained.

\subsection{Weakening the eigenvalues decay rate}

If on the one hand, the factorization \eqref{eq:SB_final_version} provides the exact form of volumetric part (that is $V_d(\varepsilon)$, the volume of the $d$--dimensional ball of radius $\varepsilon$), on the other hand, it is obtained at the cost of the hyper--exponential eigenvalues decay \eqref{eq:rate_eigen_conclusion_1}. 
Weakening the eigenvalues decay rate, the exact form of the volumetric part no longer appears.
In particular, we focus our interest on the following two further behaviours of the spectrum of $\Sigma$:
\begin{itemize}
\item \textquotedblleft super--exponential\textquotedblright :
\begin{equation*}
\lambda_d^{-1}\sum_{j\geq d+1}\lambda _{j}  =o\left( 1\right) ,\qquad \text{\textnormal{as }}%
d\rightarrow \infty .
\end{equation*}
or, equivalently,
\begin{equation} \label{eq:super-exp_decay}
\lambda _{d+1}/\lambda _{d}\rightarrow 0,\qquad \text{\textnormal{as}}\qquad
d\rightarrow \infty. 
\end{equation}%

\item \textquotedblleft exponential\textquotedblright : there exists a
positive constant $C$ so that
\begin{equation}
\lambda _{d}^{-1}\sum_{j\geq d+1}\lambda _{j}<C,\qquad \text{\textnormal{for
any }}d\in \mathbb{N}.  \label{eq:exp_decay}
\end{equation}
\end{itemize}
It is possible to show that \eqref{eq:rate_eigen_conclusion_1} $\Rightarrow$ \eqref{eq:super-exp_decay} $\Rightarrow$ \eqref{eq:exp_decay} but the contraries do not hold. For instance, for any $\alpha >1$ and $\beta >0$,  $\lambda_{j} = \exp\left\{-\beta j \right\}$ decays exponentially but not super--exponentially, $\lambda_{j} = \exp\left\{-\beta j \ln\left( \ln\left(j \right)\right) \right\}$ decays super--exponentially but not hyper--exponentially while $\lambda_{j} = \exp\left\{-\beta j^\alpha \right\}$ decays hyper--exponentially.

The following Theorem holds.
\begin{theorem}
\label{teo:weakening_decay} Under hypothesis of Theorem~\ref{prop:SB}, as $\varepsilon$ tends to zero, it is possible to choose $d=d(\varepsilon)$ diverging to infinity so that $\varphi (x,\varepsilon ) \sim f_d\left( x\right) \phi(\varepsilon,d)$, where
\begin{itemize}
\item in the super--exponential case
\begin{equation*}  
\phi(\varepsilon,d) = \exp\left\{ \frac{1}{2} d \left[ \log (2\pi e \varepsilon^2) - \log (d) + o(1) %
\right] \right\},
\end{equation*}

\item in the exponential case
\begin{equation*} 
\phi(\varepsilon,d) = \exp
\left\{ \frac{1}{2} d \left[ \log (2\pi e \varepsilon^2) - \log (d) +
\delta(d, \alpha) \right] \right\} ,
\end{equation*}
where $\delta(\cdot,\cdot)$ is such that $\lim_{\alpha\to\infty}
\limsup_{s\to \infty} \delta(s, \alpha) = 0$ and $\alpha$ is a parameter
chosen so that $\lambda_d^{-1}\varepsilon^2\le \alpha^2$.
\end{itemize}
\end{theorem}

In other words, if the decay rate changes also the volume factor does. In particular, $f_d(x)$ preserves the role of a \textit{surrogate intensity} whereas $\phi(\varepsilon,d)$ substitutes $V_d(\varepsilon)$ as volumetric term in the factorization. Observe that $\phi(\varepsilon,d)$ depends on terms (namely, $o(1)$ and $\delta(s,\alpha)$) that are implicitly defined and for which we just know the asymptotic behaviour. It is just the case to note that, in the exponential setting, Discussion~\ref{rem:intensity_Gaussian_case_2} about Gaussian and exponential power processes still holds with minor modifications.

\discussion\textbf{About slower eigenvalues decay rates.}\label{rem:standard_BM}
This theoretical problem is partially still open. In fact, a part from the Gaussian processes and, in particular, the Wiener one (whose eigenvalues decay arithmetically but the intensity, evaluated at smooth $x$, can be defined as illustrated in \ref{rem:wiener_case}), to the best of our knowledge, there are no other attempts to provide asymptotic factorizations for the \smbp{} of processes whose eigenvalues decay slower than exponentially.
Hence, at this stage, if no information about the probability law are available a solution is to go back to Theorem~\ref{prop:SB} in order to manage the dependence on $x$ and $\varepsilon$ in $\extra(x,\varepsilon,d)$.

\discussion\textbf{Optimal basis.} \label{rem:basis_independent} 
Although the factorization results in theorems~\ref{teo:SmBP_final} and \ref{teo:weakening_decay} are stated by using the Karhunen--Lo\`{e}ve (or PCA) basis, they hold for any orthonormal basis ordered according to the decreasing values of the variances of the projections, provided that they decay sufficiently fast. In particular, using the same notations as in \ref{rem:changing_basis}, if the sequence $\{\lambda _{j}\}_{j=1}^{\infty }$ has an exponential decay then Theorem~\ref{teo:weakening_decay} still holds and a \textit{surrogate intensity} can be defined.
Note that the variances obtained when one uses the PCA basis exhibit, by construction, the faster decay: in this sense the choice of this basis can be considered optimal.

\section{Estimation of the surrogate intensity}

\label{sec:joint_dist_estimate}

Besides their theoretical interest, theorems~\ref{prop:SB}, \ref{teo:SmBP_final} and \ref{teo:weakening_decay} turn to be useful from applications point of view as well. In fact, under suitable assumptions, they theoretically justify the use of $f_d$ as a \textit{surrogate intensity} for Hilbert--valued processes in statistical applications as done, for instance, within classification problems by \cite{bon:goi2016}. This fact leads immediately to the main task of this section: to make the factorization results usable for practical purposes and, in particular, to introduce an estimator of the \textit{surrogate intensity} $f_d$.

Consider a sample of random curves $\left\{ X_{i},i=1,\dots ,n\right\} $ which we
suppose i.i.d.\@ as $X$. In principle, if the sequence of eigenvalues $%
\left\{ \xi _{j}\right\} _{j=1}^{\infty }$ was known, one should consider
the empirical version of the vector of the first $d$ principal components $%
\mathbf{\theta }_{i}=\left( \theta _{1i},\dots ,\theta _{di}\right) ^{\prime
}\in \mathbb{R}^{d}$, with $\theta _{ji}=\left\langle X_{i}-\mathbb{E}\left[
X_{i}\right] ,\xi _{j}\right\rangle $, and then introduce the classical
kernel density estimate of $f_{d}$ as follows: 
\begin{equation}
f_{d,n}\left( \Pi _{d}x\right) =f_{n}\left( x\right) =\frac{1}{n}%
\sum_{i=1}^{n}K_{H_{n}}\left( \left\Vert \Pi _{d}\left( X_{i}-x\right)
\right\Vert \right)   \label{eq:estim_proj_true}
\end{equation}%
where $K_{H_{n}}\left( \mathbf{u}\right) =\det \left( H_{n}\right)
^{-1/2}K\left( H_{n}^{-1/2}\mathbf{u}\right) $, $K$ is a kernel function
and, $H_{n}=H_{nd}$ is a symmetric semi-definite positive $d\times d$ matrix
(with an abuse of notations, we dropped the dependence on $d$). In practice, the equation \eqref{eq:estim_proj_true} defines
only a pseudo-estimate for $f_{d}$: indeed, the covariance operator $\Sigma $
and then the sequence $\left\{ \xi _{j}\right\} $ are unknown. Thus, to
operationalize these pseudo-estimates it is necessary to consider the
estimates $\widehat{\mathbf{\theta }}_{i}$ and $\widehat{\Pi }_{d}$ of $%
\mathbf{\theta }_{i}$ and $\Pi _{d}$ respectively. In this view, consider
the sample versions of $\mu _{X}$ and $\Sigma $, respectively:%
\begin{equation*}
\overline{X}_{n}\left( t\right) =\frac{1}{n}{\sum_{i=1}^{n}}X_{i}(t),\qquad 
\text{\textnormal{and}}\qquad \widehat{\Sigma }_{n}[\cdot ]=\frac{1}{n}{%
	\sum_{i=1}^{n}}\langle X_{i}-\overline{X}_{n},\cdot \rangle (X_{i}-\overline{%
	X}_{n}).
\end{equation*}%
The eigenelements $\left\{ \widehat{\lambda }_{j},\widehat{\xi }_{j}\right\}
_{j=1}^{\infty }$ of $\widehat{\Sigma }_{n}$ provide an estimation for $%
\left\{ \lambda _{j},\xi _{j}\right\} _{j=1}^{\infty }$ of $\Sigma $, as
well as $\langle X_{i}-\overline{X}_{n},\widehat{\xi }_{j}\rangle =\widehat{%
	\theta }_{ji}$ estimates $\theta _{ji}$ (the asymptotic behaviour of these
estimators has been widely studied; see e.g. \citealp{bos00}). Thus, plugging
these estimates (or the estimate of the eigen--projectors) in %
\eqref{eq:estim_proj_true}, we get the kernel density estimator: 
\begin{equation}
\widehat{f}_{d,n}\left( \widehat{\Pi }_{d}x\right) =\widehat{f}_{n}\left(
x\right) =\frac{1}{n}\sum_{i=1}^{n}K_{H_{n}}\left( \left\Vert \widehat{\Pi }%
_{d}\left( X_{i}-x\right) \right\Vert \right) ,\qquad \widehat{\Pi }_{d}x\in 
\mathbb{R}^{d}.  \label{eq:estim_proj_estim}
\end{equation}%
Since $\widehat{\lambda }_{1}\geq \widehat{\lambda }_{2}\geq \dots \geq 
\widehat{\lambda }_{n}\geq 0=\widehat{\lambda }_{n+1}=\dots $ one could
choose $d=n$, but in practice this is not an appropriate choice:\ the curse of
dimensionality problems in estimate a multivariate density combined with a
bad estimation of the PCs associated to the smallest eigenelements would
jeopardize the quality of estimation. Hence, a suitable dimension $d\ll n$ had to be identified. A first naive way consists in selecting the smallest $d$ for which the Fraction of Explained Variance (defined as FEV$\left( d\right) =\sum_{j\leq d}\lambda_{j}/\sum_{j\geq 1}\lambda _{j}$) exceeds a fixed threshold. Anyway the problem of selecting the dimension $d$ to be used in practice is still open and needs further developments that go beyond the scope of this paper.

If, from a computational point of view, the replacement of $\Pi _{d}$ with $\widehat{\Pi }_{d}$ in \eqref{eq:estim_proj_estim} is a natural way to manage the problem of estimating the surrogate density in practice, one may wonder if that plug-in can influence the rate of convergence of the kernel estimator, or, in other words, if using $\widehat{f}_{n}$ instead of $f_{n}$ has no effect on this rate.
\\
To answer this question, we study the behaviour of $\mathbb{E}\left[f_{d}\left( x\right) -\widehat{f}_{n}\left( x\right) \right] ^{2}$ as $n$ goes to infinity. For the sake of simplicity, we confine the study to the special case $H_{n}=h_{n}^{2}I$ where $I$ is the identity matrix, we assumed $d$  fixed and independent of the observed data, and we suppose that

\begin{assumeALDO}\label{ass:fd_p_differentiable}
	the density $f_{d}\left( x\right) $ is positive and $p$ times differentiable at $x\in \mathbb{R}^{d}$, with $p\geq 2$;
\end{assumeALDO}
	
	\begin{assumeALDO}\label{ass:bandwidth}
		the sequence of bandwidths $h_{n}$ is such that: 
	\begin{equation*}
	h_{n}\rightarrow 0\qquad \text{\textnormal{and}}\qquad \frac{nh_{n}^{d}}{%
		\log n}\rightarrow \infty \qquad \text{\textnormal{as }}n\rightarrow \infty ;
	\end{equation*}
	\end{assumeALDO}
	
	\begin{assumeALDO}\label{ass:kernel}
		the kernel $K$ is a Lipschitz, bounded, integrable density function with compact support $\left[ 0,1\right] $;
	\end{assumeALDO}
	
	\begin{assumeALDO}\label{ass:moments_of_X}
	the process $X$ satisfies the following condition: there exists	two positive constants $s$ and $\kappa $ such that for all integer $m\geq 2$,
	\begin{equation*}
	\mathbb{E}\left[ \left\Vert X-x\right\Vert ^{m}\right] \leq \frac{m!}{2}%
	s\kappa ^{m-2}.
	\end{equation*}
	\end{assumeALDO}

The hypothesis \ref{ass:fd_p_differentiable}, \ref{ass:bandwidth} and \ref{ass:kernel} are standard in the non--parametric framework, and $p\geq 2$ is required because of \ref{ass:boundedness_second_derivatives_of_f}. Moreover, condition \ref{ass:moments_of_X} holds for a wide family of processes including the Gaussian one.

Observe firstly that one can control the quadratic mean under study by
intercalating the pseudo-estimator \eqref{eq:estim_proj_estim}; in fact,
thanks to the triangle inequality 
\begin{equation}
\mathbb{E}\left[ f_{d}\left( x\right) -\widehat{f}_{n}\left( x\right) \right]
^{2}\leq \mathbb{E}\left[ f_{d}\left( x\right) -f_{n}\left( x\right) \right]
^{2}+\mathbb{E}\left[ f_{n}\left( x\right) -\widehat{f}_{n}\left( x\right) %
\right] ^{2}.  \label{eq:disug_triang_media_2}
\end{equation}%
About the first term on the right--hand side of %
\eqref{eq:disug_triang_media_2}, it is well known in the literature (see for
instance \citealp{wan:jon95}) that under assumptions~\ref{ass:fd_p_differentiable}, \ldots, \ref{ass:moments_of_X}, and taking the
optimal bandwidth
\begin{equation}
c_{1}n^{-{1}/{(2p+d)}}\leq h_{n}\leq c_{2}n^{-{1}/{(2p+d)}}
\label{eq:optimal_bandwidth}
\end{equation}%
where $c_{1}$ and $c_{2}$ are two positive constants, one gets the minimax
rate: 
\begin{equation*}
\mathbb{E}\left[ f_{d}\left( x\right) -f_{n}\left( x\right) \right]
^{2}=O\left( n^{-2p/\left( 2p+d\right) }\right)
\end{equation*}%
uniformly in $\mathbb{R}^{d}$. Therefore, it is enough to control the second
addend on the right--hand side of \eqref{eq:disug_triang_media_2}.
\newline
The following theorem states that, assuming a suitable degree of regularity for the density $f_{d}$ depending on $d$, the rate of convergence in quadratic mean of $\widehat{f}_{n}\left( x\right) $ towards $f_{n}\left( x\right) $ is negligible with respect to the one of $f_{n}\left( x\right) $ towards $f_{d}\left( x\right) $. Thus, to use the estimated principal components instead of the empirical ones does not affect the rate of convergence.

\begin{theorem}
	\label{prop:rate_convergence} Assume \ref{ass:fd_p_differentiable}, \ldots, \ref{ass:moments_of_X} with $p> 2\vee 3d/2$ and consider
	the optimal bandwidth \eqref{eq:optimal_bandwidth}. Thus, as $n$ goes to
	infinity, 
	\begin{equation*}
	\mathbb{E}\left[ f_{n}\left( x\right) -\widehat{f}_{n}\left( x\right) \right]
	^{2}=o\left( n^{-2p/\left( 2p+d\right) }\right) ,
	\end{equation*}%
	uniformly in $\mathbb{R}^{d}$.
\end{theorem}

Formulation \eqref{eq:estim_proj_estim} requires that each random curve $%
X_{i}\left( t\right) $ is observed entirely in the continuum and without
noise over $\left[ 0,1\right] $. In practice, a discretization is
inevitable as the curves are available only at discrete design points $%
\left\{ \tau _{i,1},\tau _{i,2},\dots ,\tau _{i,p_{i}}\right\} $, $\tau
_{i,j}\in \left[ 0,1\right] $, that are not necessarily the same for each $i$%
. Thus, it is necessary to introduce some numerical approximation to compute
the estimates of the $d$ principal components involved.
\\
When each curve is observed without errors over the same fixed equispaced
grid $\left\{ \tau _{1}=0,\tau _{2},\dots ,\tau _{p-1},\tau _{p}=1\right\} $%
, with $p$ sufficiently large, then one can replace simply integrals by
summations: the empirical covariance operator is approximate by a matrix and
its eigenelements are computed by standard numerical algorithms (see \citealp{ric:sil91}). This is the approach
we follow in the simulations in Section \ref{sec:applications} below.
\\
A more general situation occurs when observed data are discretely sampled
and corrupted by noise. In this case, one has the observed pairs $\left\{
\left( \tau _{i,j},Y_{i,j}\right) ,i=1,\dots ,n,j=1,\dots ,p_{i}\right\} $,
where $Y_{i,j}=X_{i}\left( \tau _{i,j}\right) +\varepsilon _{ij}$ and the
errors $\varepsilon _{ij}$ are i.i.d.~with zero mean and finite variance. If
each $p_{i}\geq M_{n}$, where $M_{n}$ is a suitable sequence tending to
infinity with $n$ (we refer to this case as \emph{dense functional data}), a
presmoothing process is run before to conduct PCA using the sample mean and
covariance computed form the smoothed curves (see, for instance, \citealp{hal:mul:wan06}).
In this case, under suitable hypothesis, the estimators of eigenelements are
root-$n$ consistent and first-order equivalent to the estimators obtained if
curves were directly observed (see \citealp[Theorem~3]{hal:mul:wan06}).

\section{Finite sample performances in estimating the surrogate density}
\label{sec:applications}

We illustrate now, through numerical examples, the feasibility of SmBP factorization approach by exploring how the proposed estimator works in a finite sample setting. We consider only two situations because of the difficulty in finding explicit expressions for the \textit{intensity}.
At first, we focus on a finite dimensional process for which the surrogate density is straightforward derived. After, we deal with the Wiener process that is one of the few infinite dimensional processes whose \textit{intensity} can be derived, as already illustrated.
In both cases, we study how the estimates behave varying the sample size and $d$. All simulations rest on the density estimator defined in \eqref{eq:estim_proj_estim}, and are performed on a suitable grid of the $d$--dimensional factor space: the algorithms are implemented in R, and exploit the function \texttt{kde} in the package \texttt{ks} (see \citealp{duo07}).

\subsection{Finite dimensional setting}\label{sec:finite_dim_experiment}

Consider a one-dimensional random process whose trajectories are defined by%
\begin{equation*}
X\left( t\right) =a\sqrt{2/\pi }\sin \left( t\right), \qquad t\in \left[
0,\pi \right],
\end{equation*}
where $a$ is a random variable with zero mean, unitary variance, density $f_{a}$ and cumulative distribution function $F_{a}$. Given $x\left( t\right) =b\sqrt{2/\pi }\sin \left( t\right) $ with $b\in \mathbb{R}$, then, for any $\varepsilon >0$, $	\varphi \left( x,\varepsilon \right) = F_{a}\left( b+\varepsilon \right) -F_{a}\left( b-\varepsilon \right)$, 
and, as $\varepsilon$ goes to zero, $\varphi \left( x,\varepsilon \right) \sim 2\varepsilon f_{a}\left( b\right) $. Such asymptotic is the same obtained from the \smbp{} factorization: since the first PC is $\theta =a$ and $x_{1}=b$,
it holds
\begin{equation*}
\varphi \left( x,\varepsilon \right) \sim f_{1}\left( x_{1}\right)
\varepsilon \pi ^{1/2}/\Gamma \left( 1/2+1\right) =2f_{a}\left( b\right)
\varepsilon,\qquad \varepsilon \rightarrow 0,
\end{equation*}
with $f_{a}$ being the \textit{intensity} of the \smbp{}.

In this framework, $f_a$ is compared with its estimates $\widehat{f}_{1,n}$ from a sample of curves, for different $x\left( t\right) $,
varying the nature of $a$ and the sample size. In practice, set $n$ and the
distribution of $a$, we generated $1000$ samples $\left\{
X_{i}\left( t\right) ,i=1,\dots ,n\right\} $, i.i.d.~as $X\left( t\right) $,
(with $n=50,100,200,500,1000$) where every curve is discretized over a mesh
consisting on $100$ equispaced points $\{t_{j}=\left( j-1\right) \pi /99$, $%
j=1,\dots ,100\}$. For each sample, we estimated the eigenfunction $\xi
\left( t\right) $, the associated PC $\theta $ and its density via kernel procedure. Besides such samples, we built a set of curves $x^{b}\left(
t\right) =b\sqrt{2/\pi }\sin \left( t\right) $ (discretized on the same grid
as $X\left( t\right) $), where $b$ is a suitable increasing sequence of real
values. The estimated density $\widehat{f}_{1,n}$ is then evaluated at the
points $\widehat{x}_{1}^{b}=\left\langle x^{b}\left( t\right) ,\widehat{\xi }%
\left( t\right) \right\rangle $ and compared with the true values $f_{a}\left( b\right) $ in term of relative mean square prediction error (RMSEP $=$ $\sum_{b}\left[ \widehat{f}_{1,n}\left( \widehat{x}_{1}^{b}\right) -f_{a}\left( b\right) \right] ^{2}\ /\ \sum_{b}f_{a}^{2}\left( b\right) $ ) over the $1000$ replications. Moreover
it is also possible investigate for what values $b$ the estimate of the
surrogate density is better, by using the absolute percentage error (APE = $%
\left\vert \widehat{f}_{1,n}\left( \widehat{x}_{1}^{b}\right) -f_{a}\left(
b\right) \right\vert /f_{a}\left( b\right) $ ).

In the experiment we take $a$ distributed as: 1) standard Gaussian (that is, 
$a\sim \mathcal{N}\left( 0,1\right) $), 2) a standardized Student $t$
with $5$ df (that is, $a\sim t\left( 5\right) /\sqrt{5/3}$), 3) a
standardized Chi-square distribution with $8$ df (that is, $a\sim \left(
\chi ^{2}\left( 8\right) -8\right) /4$). About $b$, we used sequences
consisting of $160$ equispaced points, over the interval $\left[ -4,4\right] 
$ for the distributions 1) and 2), and $\left[ -2,6\right] $ for the
asymmetric distribution 3).

The MSEP ($\times 10^{-2}$) obtained under the different experimental
conditions are collected in Table \ref{tab:Exp1_MSEP}. As expected, results
improve\ as the sample size increases: that is due both to the better
estimates of projections $\widehat{\theta }$ and $\widehat{x}^{b}$ and to
the better performances of the kernel estimator. On the other hand,
differences due to the shape of distributions occur: long tails and
asymmetries produce a deterioration in estimates.

\begin{table}[tb!]
	\begin{center}
		\begin{tabular}{|c|cc|cc|cc|}
			\hline
			& \multicolumn{2}{c|}{$N\left( 0,1\right) $} & \multicolumn{2}{c|}{$t\left(
				5\right) /\sqrt{5/3}$} & \multicolumn{2}{c|}{$\left( \chi ^{2}\left(
				8\right) -8\right) /4$} \\ 
			$n$ & Mean & Std. & Mean & Std. & Mean & Std. \\ \hline
			50 & 3.235 & (2.681) & 5.921 & (2.557) & 4.081 & (2.842) \\ 
			100 & 1.860 & (1.444) & 4.775 & (1.503) & 2.401 & (1.619) \\ 
			200 & 1.091 & (0.824) & 4.138 & (0.878) & 1.422 & (0.887) \\ 
			500 & 0.546 & (0.355) & 3.737 & (0.477) & 0.753 & (0.443) \\ 
			1000 & 0.330 & (0.220) & 3.606 & (0.327) & 0.453 & (0.233) \\ \hline
		\end{tabular}%
	\end{center}
	\caption{Mean and standard deviation of RMSEP (Results $\times 10^{-2}$) for
		Gaussian, $t$ and $\protect\chi ^{2}$ distributions, computed over $1000$
		Monte Carlo replications varying the sample size $n$. }
	\label{tab:Exp1_MSEP}
\end{table}

The APE for some selected values $b$ when $n=200$ are reproduced in Figure %
\ref{fig:Exp1_APE}. As one can expect, the quality of estimate worsens at
the edges of the distributions, when $b$ is rather far from zero. This fact
is connected to the limitations of kernel density estimator in evaluate the
tails of distributions.

\begin{figure}[tb!]
	\begin{center}
		\includegraphics[trim = 1cm 1cm 1cm 1cm,
		height=5cm,width=5cm]{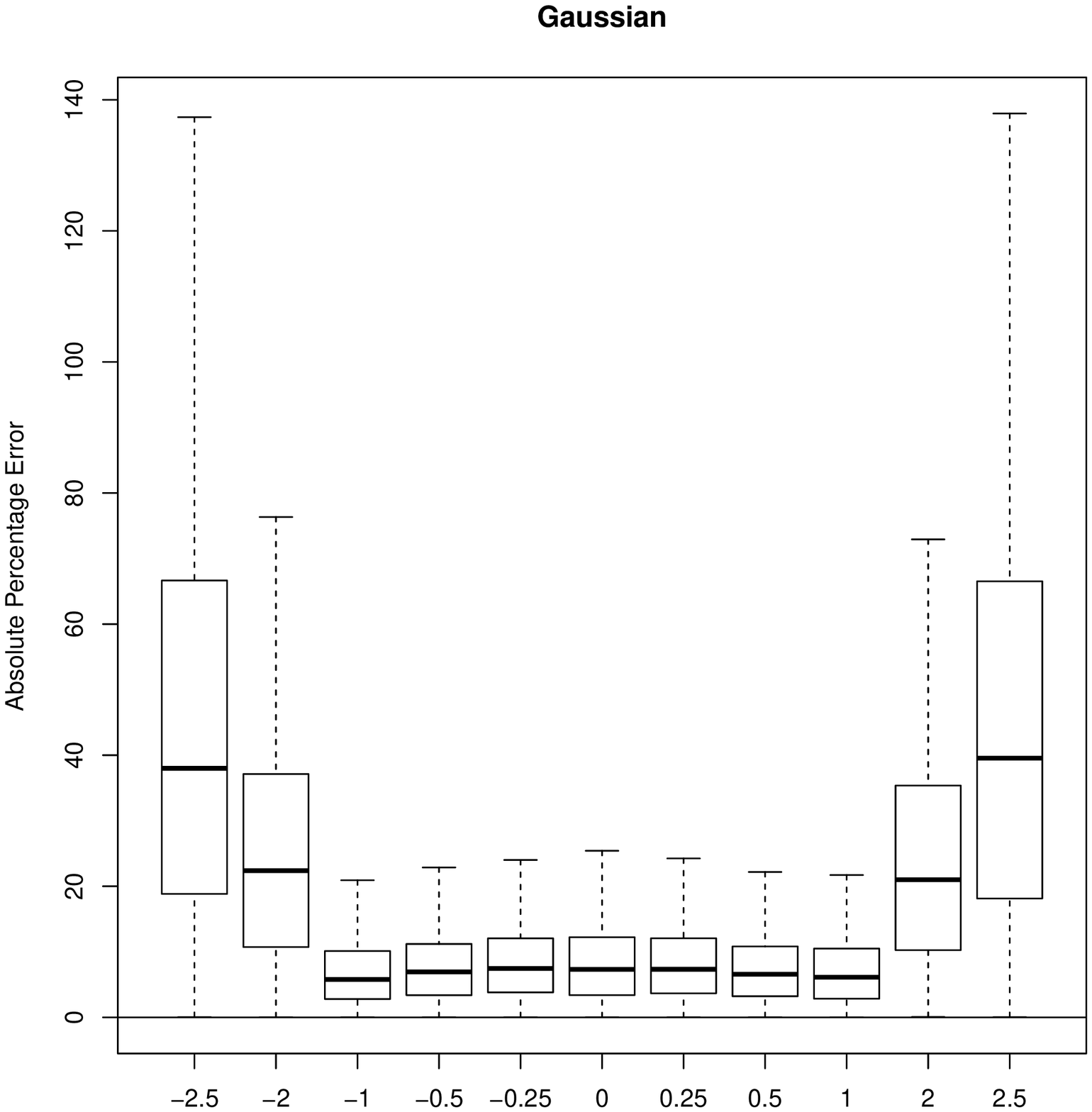} 
		\includegraphics[trim = 0cm 1cm 1cm 1cm,
		height=5cm,width=5cm]{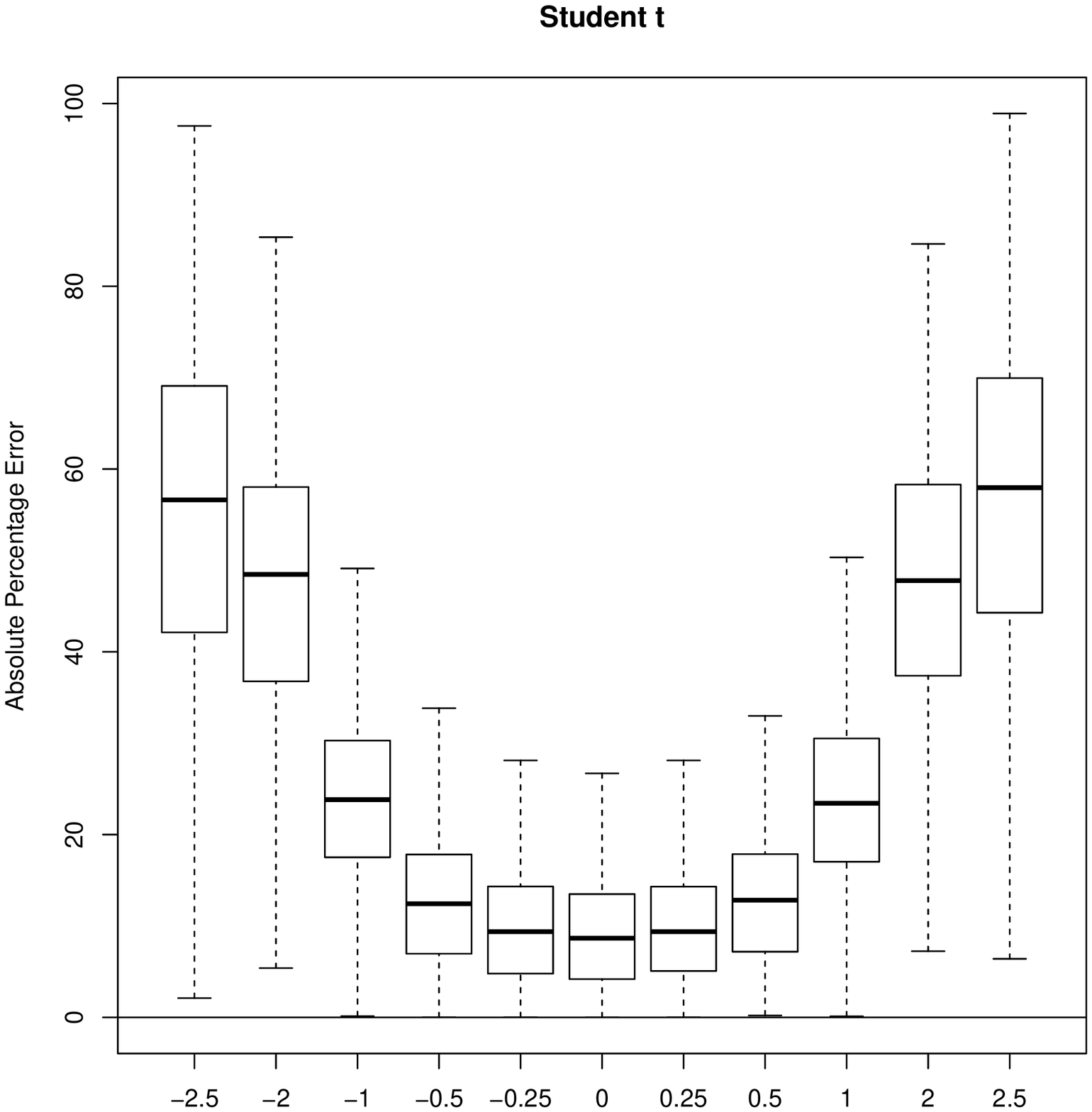} 
		\includegraphics[trim =
		0cm 1cm 1cm 1cm, height=5cm,width=5cm]{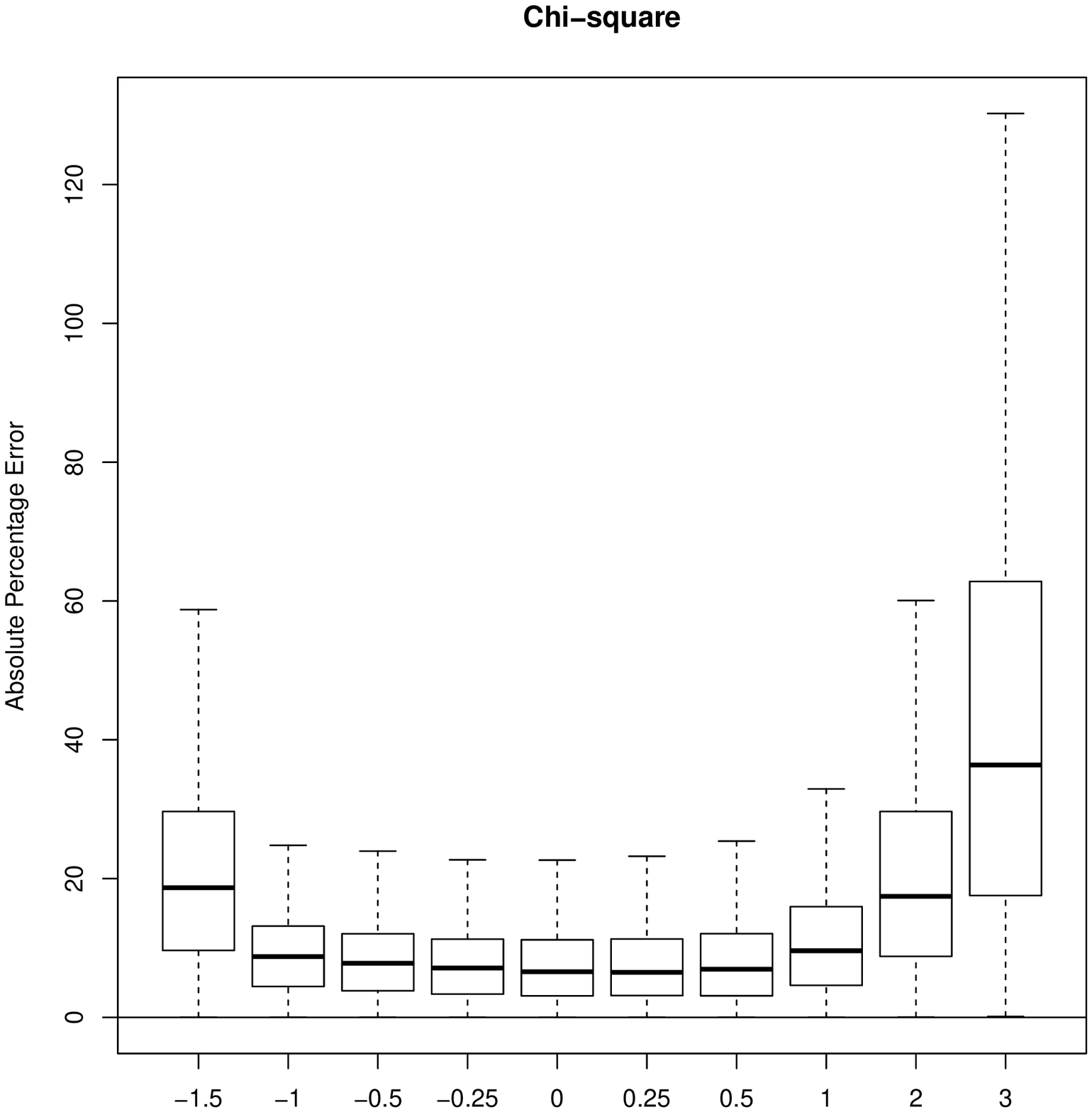}
	\end{center}
	\caption{Absolute percentage errors in estimating $f_a\left( b\right) $
		varying $b$ for Normal, $t$ and $\protect\chi ^{2}$ distributions respectively.}
	\label{fig:Exp1_APE}
\end{figure}

\subsection{Infinite dimensional setting}
\label{sec:applications_MB}

In this second experiment, we deal with an infinite dimensional setting in order to study how the estimation of the \textit{intensity} of the \smbp{} behaves according to the sample size and the dimensional parameter $d$. To do this, keeping in mind \ref{rem:wiener_case}, let us consider a Wiener process $X$ on $\left[ 0,1\right] $ and the smooth function
$
	x\left( t\right) =\sum_{j=1}^{d_0}b_{j}\xi _{j}\left(t\right),
$
with, for the sake of simplicity, $d_0=1$, that is 
\begin{equation}
	x\left( t\right) =b\frac{2\sqrt{2}}{\pi }\sin \left( \frac{\pi t}{2}\right),\qquad t\in \left[ 0,1\right] ,  \label{eq:x_b}
\end{equation}%
where $b\in \mathbb{R}$, so that the \textit{intensity} is $\Psi_{d_0}(x) = \exp \left\{ -b^{2}/2\right\}$.

The experiment follows a similar route as in \ref{sec:finite_dim_experiment}. In a first step, we generated $1000$ samples $\left\{ X_{i}\left( t\right) ,i=1,\dots ,n\right\} $, (with $n=50,100,200,500,1000$) where every curve is discretized over $100$
equispaced points $\mathcal{G}=\left\{ t_{j}=\left( j-1\right) /99,j=1,\dots
,100\right\} $, and $160$ fixed curves $x^{b}\left( t\right) $ generated
according to (\ref{eq:x_b}) and discretized over $\mathcal{G}$ ($b$ is an
increasing sequence of equispaced points, over the interval $\left[ -4,4%
\right] $). In a second step, for each sample, once empirical eigenfunctions 
$\widehat{\xi }_{j}\left( t\right) $ are obtained, we estimated $f_{d}$
(with $d=1,\dots ,6$) and we computed them at $\left( \widehat{x}%
_{1}^{b},\dots ,\widehat{x}_{d}^{b}\right) ^{\prime }$ where $\widehat{x}%
_{j}^{b}=\left\langle x^{b}\left( t\right) ,\widehat{\xi }_{j}\left(
t\right) \right\rangle $. Finally, we compared the estimated surrogate density
with the true one in term of relative mean square prediction error (MSEP)
over the $1000$ replications. The obtained results, varying $n$ and $d$ are
reported in Table \ref{tab:Exp2_MSEP}.

\begin{table}[tbp]
	\begin{center}
		{\scriptsize 
			\begin{tabular}{|c|cccccc|}
				\hline
				$n$ & $d=1$ & $d=2$ & $d=3$ & $d=4$ & $d=5$ & $d=6$ \\ \hline
				50 & 3.361 (2.506) & 7.197 (3.728) & 13.529 (7.338) & 22.029 (12.052) & 
				31.900 (15.871) & 42.051 (19.159) \\ 
				100 & 1.948 (1.198) & 4.821 (2.585) & 9.471 (5.535) & 15.857 (8.711) & 
				23.988 (12.272) & 33.729 (15.871) \\ 
				200 & 1.158 (0.719) & 3.143 (1.601) & 6.638 (3.773) & 11.514 (6.295) & 
				17.894 (9.476) & 25.514 (13.098) \\ 
				500 & 0.574 (0.334) & 1.775 (0.932) & 4.168 (2.363) & 7.766 (4.232) & 12.956
				(6.884) & 18.988 (9.284) \\ 
				1000 & 0.345 (0.187) & 1.149 (0.625) & 2.822 (1.636) & 5.858 (3.126) & 
				10.091 (5.425) & 15.292 (7.651) \\ \hline
			\end{tabular}%
		}
	\end{center}
	\caption{Mean and standard deviation (in parentheses) of RMSEP (Results $%
		\times 10^{-2}$) for Wiener process, computed over $1000$ Monte
		Carlo replications varying the sample size $n$ and the dimension $d$.}
	\label{tab:Exp2_MSEP}
\end{table}

As a general comment, one can observe that for each $d$ the MSPE reduces
(both in mean and in variability) increasing $n$, whereas, for each $n$, the
MSPE increases (both in mean and in variability) with $d$. This shows how
the curse of dimensionality interferes in the kernel estimation procedure as
soon as the dimension $d$ exceeds one. To perceive the relation between $d$
and $n$, one has to read the table in diagonal sense: it is possible to use
large $d$ at the cost of use large samples. For instance, similar very good
results (at around $3\%$) are possible using $n=50$ and $d=1$, or $n=200$
and $d=2$, or when $n=1000$ and $d=3$. On the other hand, results benefit
from the fact that the spectrum of the process is rather concentrate. In
fact, the Fraction of Explained Variance (defined as FEV$\left( d\right)
=\sum_{j\leq d}\lambda _{j}/\sum_{j\geq 1}\lambda _{j}$) are: FEV$\left(
1\right) =0.811$, FEV$\left( 2\right) =0.901$, FEV$\left( 3\right) =0.933$,
FEV$\left( 4\right) =0.950$, FEV$\left( 5\right) =0.960$ and FEV$\left(
6\right) =0.966$. Hence, good estimates for the surrogate density are already
possible with $d=1$ or $d=2$, also for medium size samples. In that sense,
this experiment gives an empirical evidence on the use of FEV in select the
dimensional parameter $d$, conscious of the fact that large $d$ need large $%
n $ in order to have better estimates.

\section{Proofs}

\label{sec:proofs}

This section collects proofs of results exposed above.

\subsection{Proof of Theorem~\protect\ref{prop:SB}}

We are interested in the asymptotic behaviour, whenever $\varepsilon $ tends
to zero, of the \smbp{} of the process $X$, that is
\begin{align*}
\varphi (x,\varepsilon ) & =\mathbb{P}\left( \left\Vert X-x\right\Vert \leq
\varepsilon \right) =\mathbb{P}\left( \left\Vert X-x\right\Vert ^{2}\leq
\varepsilon ^{2}\right) \\
&=\mathbb{P}\left( \sum_{j=1}^{+\infty } \langle X-x, \xi_j \rangle^{2}\leq
\varepsilon ^{2}\right) =\mathbb{P}\left( \sum_{j=1}^{+\infty }
\left(\theta_j - x_j \right)^{2}\leq \varepsilon ^{2}\right) ,\qquad \text{%
\textnormal{as }}\varepsilon \rightarrow 0
\end{align*}
Let $S_{1}=\sum_{j\leq d}\left(\theta_j - x_j \right)^{2}$ and $S=\frac{1}{%
\varepsilon^{2}} \sum_{j\geq d+1} \left(\theta_j - x_j \right)^{2}$ be the
truncated series and the scaled version of the remainder respectively. Thus,
the \smbp{} is
\begin{align}
\varphi (x,\varepsilon ) & = \mathbb{P}\left( S_{1}+\varepsilon ^{2}S\leq
\varepsilon^{2}\right) =\mathbb{P}\left( S_{1}\leq \varepsilon ^{2}\left(
1-S\right) \right)  \notag 
\\
&= \mathbb{P}\left( \left\{ S_{1}\leq \varepsilon ^{2}\left( 1-S\right)
\right\} \cap \left\{ S \ge 1 \right\}\right) + \mathbb{P}\left( \left\{
S_{1}\leq \varepsilon ^{2}\left( 1-S\right) \right\} \cap \left\{ 0\le S <
1 \right\}\right)  \notag \\
& =\mathbb{P}\left( \left\{ S_{1}\leq \varepsilon ^{2}\left( 1-S\right)
\right\} \cap \left\{ 0\le S < 1 \right\}\right)  \notag \\
&=\int_{0}^{1} \varphi(s|x,\varepsilon,d)dG\left( s\right)
\label{eq:SmBP_as_integral_of_reduced_smbp}
\end{align}
where $G$ is the cumulative distribution function of $S$. At first, for any $%
s\in (0,1)$, let us consider $\varphi(s|x,\varepsilon,d)$, that is the
\smbp{} about $\Pi _{d}x$ of the process $\Pi _{d}X$ in the space spanned by $%
\left\{\xi _{j}\right\} _{j\leq d}$. In terms of $f_d\left( \cdot \right) $,
the probability density function of $\boldsymbol{\vartheta}=\left(
\vartheta_{1}, \ldots ,\vartheta_{d}\right) ^{\prime }$, it can be written
as
\begin{equation*}
\varphi(s|x,\varepsilon,d)=\int_{D^x}f_d\left( \boldsymbol{\vartheta}\right) d%
\boldsymbol{\vartheta},
\end{equation*}
where 
$
	D=D^x= \left\{ \boldsymbol{\vartheta}\in \mathbb{R}^{d}:\sum_{j\leq d}\left(\vartheta_j - x_j \right)^{2}\leq \varepsilon ^{2}\left( 1-s\right) \right\} 
$
is a $d$--dimensional ball centered about $\Pi_{d}x = (x_1,\ldots,x_d)$ with radius $\varepsilon \sqrt{1-s}$. Now, consider the Taylor expansion of $f=f_d$ about $\Pi x=\Pi _{d}x$,
\begin{align*}
f\left( \boldsymbol{\vartheta} \right) = & f(x_1,\ldots,x_d) + \left\langle
\boldsymbol{\vartheta}-\Pi x , \nabla f(x_1,\ldots,x_d) \right\rangle \\
&+\frac{1}{2} (\boldsymbol{\vartheta}-\Pi x)^{\prime }H_{f}\left(\Pi x+(%
\boldsymbol{\vartheta}-\Pi x) t \right) (\boldsymbol{\vartheta}-\Pi x),
\end{align*}
for some $t \in \left( 0,1\right)$ and with $H_{f}$ denoting the Hessian
matrix of $f$. (In general, $t$ depends on $\boldsymbol{\vartheta}-\Pi x$,
but we are not interested in the actual value of it because the boundedness
of the second derivatives of $f$ allows us to drop, in what follows, those
terms depending on $t$). Then we can write
\begin{align}
\varphi(s|x,\varepsilon,d) = & \int_{D}\left( f(x_1,\ldots,x_d) +
\left\langle \boldsymbol{\vartheta}-\Pi x , \nabla f (x_1,\ldots,x_d)
\right\rangle \phantom{\frac{1}{1}} \right.  \notag \\
&\left. +\frac{1}{2} (\boldsymbol{\vartheta}-\Pi x)^{\prime }H_{f}\left(\Pi
x+(\boldsymbol{\vartheta}-\Pi x) t \right) (\boldsymbol{\vartheta}-\Pi x)
\right) d\boldsymbol{\vartheta}  \notag \\
= & f(x_1,\ldots,x_d) \int_{D} d\boldsymbol{\vartheta} + \int_{D}
\left\langle \boldsymbol{\vartheta}-\Pi x , \nabla f(x_1,\ldots,x_d)
\right\rangle d\boldsymbol{\vartheta}  \notag \\
& + \frac{1}{2} \int_{D} (\boldsymbol{\vartheta}-\Pi x)^{\prime
}H_{f}\left(\Pi x+(\boldsymbol{\vartheta}-\Pi x) t \right) (\boldsymbol{%
\vartheta}-\Pi x) d\boldsymbol{\vartheta}  \notag \\
= & f(x_1,\ldots,x_d) I + \frac{1}{2} \int_{D} (\boldsymbol{\vartheta}-\Pi
x)^{\prime }H_{f}\left(\Pi x+(\boldsymbol{\vartheta}-\Pi x) t \right) (%
\boldsymbol{\vartheta}-\Pi x) d\boldsymbol{\vartheta}
\label{eq:approx_SmBP_primo_ordine_proof_1}
\end{align}
where $I=I\left( s,\varepsilon, d\right) $ denotes the volume of $D$ that is
\begin{equation*}
I= \frac{\varepsilon^d \pi ^{d/2}}{\Gamma \left(d/2+1\right) }\left( 1-s\right) ^{d/2}
\end{equation*}
and, the addend $\int_{D} \left\langle \boldsymbol{\vartheta}-\Pi x , \nabla
f(x_1,\ldots,x_d) \right\rangle d\boldsymbol{\vartheta}$ is null since $%
\left\langle \boldsymbol{\vartheta}-\Pi x , \nabla f(x_1,\ldots,x_d)
\right\rangle $ is a linear functional integrated over the symmetric -- with
respect to the center $(x_1,\ldots,x_d)$ -- domain $D$. Thus from \eqref{eq:approx_SmBP_primo_ordine_proof_1}, thanks to:  the boundedness of second derivatives \eqref{eq:boundedness_second_derivatives_of_f}, the fact that symmetry arguments lead to $\int_{D}(\vartheta_i-x_i)(\vartheta_j-x_j)d\boldsymbol{\vartheta}=0$ for $i\neq j$ and monotonicity of eigenvalues, it follows
\begin{align*}
\vert\varphi(s|x,\varepsilon,d) & -f(x_1,\ldots,x_d) I \vert = \\
&= \left\vert \frac{1}{2} \int_{D}\sum_{i\leq d}\sum_{j\leq
d}(\vartheta_i-x_i)(\vartheta_j-x_j)\frac{\partial^{2}f}{\partial
\vartheta_{i}\partial\vartheta_{j}}\left(\Pi x+ (\boldsymbol{\vartheta}-\Pi
x) t \right) d\boldsymbol{\vartheta}\right\vert \\
&\leq \frac{1}{2}C_2 f(x_1,\ldots,x_d) \left\vert \sum_{i\leq d}\sum_{j\leq
d} \int_{D} \frac{(\vartheta_i-x_i)(\vartheta_j-x_j)}{\sqrt{\lambda_i}\sqrt{%
\lambda_j}} d\boldsymbol{\vartheta}\right\vert \\
&=\frac{1}{2}C_2 f(x_1,\ldots,x_d) \int_{D} \sum_{j\leq d} \frac{%
(\vartheta_j-x_j)^2}{\lambda_j} d\boldsymbol{\vartheta}
\\
&\le \frac{C_2 }{2\lambda_d}f(x_1,\ldots,x_d) \int_{D} \sum_{j\leq d} (\vartheta_j-x_j)^2 d\boldsymbol{\vartheta}.
\end{align*}
Note that
\begin{equation*}
\int_{D} \sum_{j\leq d} (\vartheta_j-x_j)^2 d\boldsymbol{\vartheta} = \int_{\|\boldsymbol{\vartheta}\|_{\mathbb{R}^d}^2 \le \varepsilon^{2}(1-s)} \|\boldsymbol{\vartheta}\|_{\mathbb{R}^d}^2 d\boldsymbol{\vartheta}
\end{equation*}
whose integrand is a radial function (i.e.\@ a map $H:\mathbb{R}^d \to \mathbb{R}$ such that $H(\boldsymbol{\vartheta}) = h(\|\boldsymbol{\vartheta}\|_{\mathbb{R}^d})$ with $h:\mathbb{R}\to\mathbb{R}$), for which the following identity applies
\begin{equation*}
	\int_{\|\boldsymbol{\vartheta}\|_{\mathbb{R}^d}\le R} H(\boldsymbol{\vartheta}) d\boldsymbol{\vartheta} = \omega_{d-1} \int_0^R h(\rho) \rho^{d-1} d\rho,
\end{equation*} 
where $\omega_{d-1}$ denotes the surface area of the sphere of radius 1 in $\mathbb{R}^d$.
Hence 
\begin{equation*}
\int_{\|\boldsymbol{\vartheta}\|_{\mathbb{R}^d}^2 \le \varepsilon^{2}(1-s)} \|\boldsymbol{\vartheta}\|_{\mathbb{R}^d}^2 d\boldsymbol{\vartheta}
=  \frac{2 \pi^{d/2} }{\Gamma \left(d/2\right)} \int_{0}^{\varepsilon\sqrt{1-s}} \rho^{d+1} d\rho
= \frac{ d }{(d+2)} I\varepsilon^{2}(1-s) \le I\varepsilon^{2},
\end{equation*}
where the latter inequality follows from the fact that $s\in[0,1)$. 
This leads to 
\begin{equation}
\left\vert \varphi(s|x,\varepsilon,d) -f(x_1,\ldots,x_d) I \right\vert \le C_2  \frac{\varepsilon^2 I}{2\lambda_d}f(x_1,\ldots,x_d) .  \label{eq:boundeness_first_order_approx_varphi(t)}
\end{equation}
Come back to the \smbp{} \eqref{eq:SmBP_as_integral_of_reduced_smbp},
\begin{equation}  \label{eq:SmBP_as_addition_of_dominant_plus_o_small}
\varphi (x,\varepsilon ) =\int_{0}^{1}f(x_1,\ldots,x_d) I dG\left( s\right)
+ \int_{0}^{1}\left( \varphi(s|x,\varepsilon,d)-f(x_1,\ldots,x_d) I \right)
dG\left( s\right) ,
\end{equation}
and note that, thanks to \eqref{eq:boundeness_first_order_approx_varphi(t)}
and because $d$ is fixed, the second addend in the right--hand side of %
\eqref{eq:SmBP_as_addition_of_dominant_plus_o_small} is infinitesimal with
respect to the first addend
\begin{align*}
& \left\vert \frac{ \int_{0}^{1}\left(
\varphi(s|x,\varepsilon,d)-f(x_1,\ldots,x_d) I \right) dG\left( s\right) }{%
\int_{0}^{1}f(x_1,\ldots,x_d) I dG\left( s\right) } \right\vert \leq \\
& \leq \left\vert \frac{C_2  \frac{\varepsilon^2}{2\lambda_d}f(x_1,\ldots,x_d)
\int_{0}^{1}I dG\left( s\right) }{f(x_1,\ldots,x_d) \int_{0}^{1}I dG\left( s\right)}\right\vert =
C_2  \frac{ \varepsilon^{2}}{2\lambda_d}  .
\end{align*}
Noting that
\begin{equation*}
\int_{0}^{1}I(s,\varepsilon,d) dG(s) = \frac{\varepsilon^d \pi ^{d/2}}{%
\Gamma \left( d/2+1\right) }\mathbb{E}\left[ \left(1-S\right)^{d/2}\mathbb{I}%
_{\left\{ S\leq 1\right\} }\right],
\end{equation*}
we obtain
\begin{equation*}
\left|\varphi (x,\varepsilon ) - \varphi_d(x,\varepsilon) \right| \le C_2  \frac{\varepsilon^{2}}{2\lambda_d} %
\varphi_d(x,\varepsilon) \tag{\ref{eq:SmBP_inequality_r_fixed}}
\end{equation*}
where,
\begin{equation*}
\varphi_d(x,\varepsilon) = f(x_1,\ldots,x_d) \frac{\varepsilon^{d}\pi ^{d/2}%
}{\Gamma \left( d/2+1\right) } \mathbb{E}\left[ \left(1-S\right)^{d/2}%
\mathbb{I}_{\left\{ S\leq 1\right\} }\right] .
\tag{\ref{eq:SmBP_approximation}}
\end{equation*}
Thus, since $d$ is fixed, as $\varepsilon $ tends to zero,
\begin{align*}
\varphi (x,\varepsilon )=\int_{0}^{1} \varphi(s|x,\varepsilon,d) dG\left(
s\right) =& \varphi_d(x,\varepsilon) + o\left( \frac{\varphi_d(x,\varepsilon)%
}{f(x_1,\ldots,x_d) }\right)
\end{align*}
or, equivalently, $\varphi (x,\varepsilon ) \sim \varphi_d(x,\varepsilon) $
that concludes the proof.

\subsection{Proofs of Proposition~\protect\ref{pro:E[(1-S)^r/2]_asymptotics}, and theorems~\ref{teo:SmBP_final} and \ref{teo:weakening_decay}}

To prove Proposition~\ref{pro:E[(1-S)^r/2]_asymptotics} we need the following Lemma.
\begin{lemma}
	\label{pro:S_convergences} Assume \ref{ass:zero_mean} and \ref{ass:boundedness_of_x}. Then, it is possible to choose $d=d(\varepsilon)$ so that it diverges to infinity as $\varepsilon $ tends to zero and
	\begin{equation}  \label{eq:r_fixed_for_S_convergence}
	\sum_{j\geq d+1}\lambda _{j} = o(\varepsilon^{2}).
	\end{equation}
	Moreover, as $\varepsilon \rightarrow 0$, $S(x,\varepsilon,d)\rightarrow 0$, where the convergence holds almost surely, in the $L^{1}$ norm and hence in probability. 
\end{lemma}
\begin{proof}
A possible choice for $d=d(\varepsilon )$ satisfying \eqref{eq:r_fixed_for_S_convergence} can be, for a fixed $\delta >0$, as follow
\begin{equation*}
d=\min \left\{ k\in \mathbb{N}:\sum_{j\geq k+1}\lambda _{j}\leq \varepsilon
^{2+\delta }\right\} ,\qquad \text{\textnormal{for any }}\varepsilon >0.
\end{equation*}%
Such a minimum is well defined since eigenvalues series is convergent.
\\
Let us prove that $S$ converges to zero in probability. For any $k>0$,
by Markov inequality and, thanks to Equation~\eqref{eq:E[W^2]_bounded},
\begin{align}
\mathbb{P}\left( \left\vert S\right\vert >k\right) &= \mathbb{P}%
\left(S>k\right) =\mathbb{P}\left( \frac{1}{\varepsilon^{2}} \sum_{j\geq
d+1} \left(\theta_j - x_j \right)^{2} > k\right)  \notag \\
&\leq \frac{\mathbb{E}\left[ \frac{1}{\varepsilon ^{2}}\sum_{j\geq d+1}
\left(\theta_j - x_j \right)^{2} \right] }{k^{2}} \leq \frac{C_1 }{k^{2}}%
\frac{\sum_{j\geq d+1}\lambda _{j}}{\varepsilon ^{2}}.
\label{eq:rate_probability_convergence_of_S}
\end{align}
Thanks to \eqref{eq:r_fixed_for_S_convergence} we get the convergence in
probability. Since $S=S(x,\varepsilon,d )$ is non--increasing when $d$
increases,
\begin{equation*}
\mathbb{P}\left( \sup_{j\geq d+1}\left\vert S(x,\varepsilon,j )-0\right\vert
\geq k\right) =\mathbb{P}\left( S(x,\varepsilon,d+1 )\geq k\right)
\end{equation*}
holds for any $k>0$ and any $x$. This fact, together with %
\eqref{eq:rate_probability_convergence_of_S}, guarantees the almost sure
convergence of $S$ to zero (e.g. \citealp[Theorem 10.3.1]{shi95}) as $%
\varepsilon$ tends to zero. Moreover, the monotone convergence theorem
guarantees the $L^{1}$ convergence.
\end{proof}

\begin{proof}[{Proof of Proposition~\protect\ref{pro:E[(1-S)^r/2]_asymptotics}%
}]
Note that if $d(\varepsilon)$ satisfies \eqref{eq:r_fixed_for_E(1-S)^r/2_convergence}, then \eqref{eq:r_fixed_for_S_convergence} and Lemma~\ref{pro:S_convergences} hold. For a fixed $\delta>0$, a possible choice of such $d=d(\varepsilon)$ can be
\begin{equation*}
d = \min\left\{k\in\mathbb{N}:\ \ k\!\!\sum_{j\geq k+1}\lambda _{j}\leq
\varepsilon^{2+\delta } \right\},
\end{equation*}
where the minimum is achieved thanks to the eigenvalues hyperbolic decay assumption.
\\
At this stage, note that
\begin{equation*}
0 < \mathbb{E}\left[ \left( 1-S\right) ^{d/2}\mathbb{I}_{\left\{ S<
1\right\}}\right] \le 1
\end{equation*}
then, after some algebra, thanks to Bernoulli inequality (i.e.\@ $(1+s)^r
\ge 1+rs$ for $s\ge -1$ and $r\in\mathbb{R}\setminus (0,1)$), Markov
inequality and Assumption~\eqref{eq:E[W^2]_bounded}, we have (for any $d\ge
2 $)
\begin{align*}
0 \le & 1- \mathbb{E}\left[ \left( 1-S\right) ^{d/2}\mathbb{I}_{\left\{
S< 1\right\}}\right] \le 1- \mathbb{E}\left[ \left( 1-\frac{d}{2}S\right)%
\mathbb{I}_{\left\{ S< 1\right\}} \right] \\
\le & \mathbb{P}(S\ge 1) + \mathbb{E}\left[ \frac{d}{2} S\, \mathbb{I}%
_{\left\{ S< 1\right\}}\right] \le \mathbb{E}\left[ \frac{(d+2)}{%
2\varepsilon ^{2}}\sum_{j\ge d+1}\left(\theta_j - x_j \right)^{2} \right]
\le \frac{C_1 (d+2)}{2\varepsilon ^{2}}\sum_{j\ge d+1}\lambda _{j}.
\end{align*}
Choosing $d$ according to \eqref{eq:r_fixed_for_E(1-S)^r/2_convergence} the
thesis follows.
\end{proof}

\begin{proof}[Proof of Theorem~\ref{teo:SmBP_final}]
	Thanks to hyper--exponentiality \eqref{eq:rate_eigen_conclusion_1}, there exists $d_{0}\in \mathbb{N}$ so that for any $d\geq d_{0}$
	\begin{equation*}
	d\sum_{j\geq d+1}\lambda _{j}< \lambda _{d}.
	\end{equation*}%
	Moreover, there exist $\delta _{1},\delta _{2}\in (0,1)$ (depending on $d$)
	for which, for any $d\geq d_{0}$
	\begin{equation}
	0\leq d\sum_{j\geq d+1}\lambda _{j}\leq b(d,\{\lambda _{j}\}_{j\geq
		d+1},\delta _{1})<B(d,\{\lambda _{j}\}_{j\leq d},\delta _{2})\leq \lambda _{d},  \label{eq:b<B}
	\end{equation}%
	where
	\begin{equation*}
	b(d,\{\lambda _{j}\}_{j\geq d+1},\delta _{1})=\left( d\sum_{j\geq
		d+1}\lambda _{j}\right) ^{1-\delta _{1}},\qquad B(d,\{\lambda _{j}\}_{j\leq
		d},\delta _{2})=\lambda _{d}^{1-\delta
		_{2}} .
	\end{equation*}%
	As instance, for a given $d\geq d_{0}$, fix $\delta _{1}\in (0,1)$ and solve %
	\eqref{eq:b<B} with respect to $\delta _{2}$, that is $\delta _{2}\in \left(
	\min \left\{ 0,\beta (\delta _{1})\right\} ,1\right) $ where $\beta (\delta
	_{1})=1-(1-\delta _{1})\ln \left( d\sum_{j\geq d+1}\lambda _{j}\right) /\ln
	\left( \lambda _{d}\right) $. As a consequence, for any $%
	\varepsilon >0$ and for such a choice of $\delta _{1}$, $\delta _{2}$, the
	following minimum is well--defined
	\begin{equation*}
	d(\varepsilon )=\min \left\{ k\in \mathbb{N}:b(k,\{\lambda _{j}\}_{j\geq
		k+1},\delta _{1})\leq \varepsilon ^{2}\leq B(k,\{\lambda _{j}\}_{j\leq
		k},\delta _{2})\right\} .  
	\end{equation*}%
	This guarantees that the right--hand side of \eqref{eq:errore_globale}
	vanishes as $\varepsilon $ goes to zero.
\end{proof}

To prove Theorem~\ref{teo:weakening_decay} we need the following Lemma.
\begin{lemma}
		Assume \ref{ass:zero_mean} and \ref{ass:boundedness_of_x}. Then, as $\varepsilon \rightarrow 0$,
		\begin{equation}  \label{eq:DH_6.17}
		\extra(x,\varepsilon,d)^{2/d} \to 1, \qquad \textnormal{or,}\qquad \log\left( \extra(x,\varepsilon,d) \right) = o(d).
		\end{equation}
\end{lemma}
\begin{proof}
	Jensen inequality for concave functions (i.e. $\mathbb{E} [f(g)]\le f (\mathbb{E} [g])$ if $f$ is a concave function) guarantees that 
	\begin{align*}
	\mathbb{E}\left[ \left((1-S)\mathbb{I}_{\{S< 1\}}\right)^{\frac{d}{2}} \right]
	= &\ \mathbb{E}\left[ \left((1-S) \mathbb{I}_{\{S< 1\}}\right)^{\frac{d+1}{2}\frac{d}{d+1}} \right] 
	\\
	\le & \left\{ \mathbb{E}\left[ \left((1-S)\mathbb{I}_{\{S< 1\}}\right)^{\frac{d+1}{2}} \right]\right\}^{\frac{d}{d+1}},
	\end{align*}
	noting that $S(x,\varepsilon,d+1) =: S_{d+1} \le S_d := S(x,\varepsilon,d)$ and $\mathbb{I}_{\{S_d< 1\}} \le \mathbb{I}_{\{S_{d+1}< 1\}}$, then
	\begin{equation*}
	\mathbb{E}\left[ \left((1-S_d)\mathbb{I}_{\{S_d< 1\}}\right)^{\frac{d}{2}} \right]
	\le 
	\left\{ \mathbb{E}\left[ \left((1-S_{d+1})1_{\{S_{d+1}< 1\}}\right)^{\frac{d+1}{2}} \right]\right\}^{\frac{d}{d+1}}.
	\end{equation*}
	The latter guarantees that $\mathbb{E}\left[ \left( 1-S\right) ^{d/2}\mathbb{I}_{\left\{ S< 1\right\}}\right]^{2/d}$ is a non--decreasing monotone sequence with respect to $d$ whose values are in $(0,1]$ and eventually bounded away from zero.
\end{proof}

\begin{proof}[Proof of Theorem~\ref{teo:weakening_decay}]
	Given results in Theorem~\ref{prop:SB}, thesis holds using same arguments as
	in \cite[Proof of Theorem~4.2.]{del:hal10}: the idea is to combine together %
	\eqref{eq:DH_6.17}, the Stirling expansion of the Gamma function in $V_d$
	and the (super--)exponential eigenvalues decay.
\end{proof}

\subsection{Proof of Theorem \protect\ref{prop:rate_convergence}}

In what follows, as did in Section~\ref{sec:joint_dist_estimate}, we
simplify the notations dropping the dependence on $d$ for the density
estimators $f_{n}$ and $\widehat{f}_{n}$. Moreover, $C$ denotes a
general positive constant. The proof of Theorem \ref{prop:rate_convergence} uses similar arguments as in \cite{bia:mas12}.
\newline
Since $H_{n}=h_{n}^{2}I$, it holds $K_{H_{n}}\left( u\right)
=h_{n}^{-d}K\left( u\right) $. Consider
\begin{equation*}
S_{n}\left( x\right) =\sum_{i=1}^{n}K\left( \frac{\left\Vert \Pi _{d}\left(
	X_{i}-x\right) \right\Vert }{h_{n}}\right), \qquad \widehat{S}_{n}\left( x\right) =\sum_{i=1}^{n}K\left( \frac{%
	\left\Vert \widehat{\Pi }_{d}\left( X_{i}-x\right) \right\Vert }{h_{n}}%
\right),
\end{equation*}%
then the pseudo-estimator and the estimator are given by
\begin{equation*}
f_{n}\left( x\right) =\frac{S_{n}\left( x\right) }{nh_{n}^{d}}, \qquad \widehat{f}_{n}\left( x\right) =\frac{\widehat{S}%
	_{n}\left( x\right) }{nh_{n}^{d}},
\end{equation*}%
and, hence,
\begin{equation*}
\mathbb{E}\left[ f_{n}\left( x\right) -\widehat{f}_{n}\left( x\right) \right]
^{2}=\frac{1}{\left( nh_{n}^{d}\right) ^{2}}\mathbb{E}\left[ S_{n}\left(
x\right) -\widehat{S}_{n}\left( x\right) \right] ^{2}.
\end{equation*}%
Set $V_{i}=\left\Vert \Pi _{d}\left( X_{i}-x\right) \right\Vert $, $\widehat{V}_{i}=\left\Vert \widehat{\Pi }_{d}\left( X_{i}-x\right) \right\Vert $, consider the events 
\begin{equation*}
A_{i}=\left\{ V_{i}\leq h_{n}\right\} ,\qquad B_{i}=\left\{ \widehat{V}%
_{i}\leq h_{n}\right\},
\end{equation*}%
then we have the decomposition
\begin{align*}
S_{n}\left( x\right) -\widehat{S}_{n}\left( x\right) =& \sum_{i=1}^{n}\left[
K\left( \frac{V_{i}}{h_{n}}\right) -K\left( \frac{\widehat{V}_{i}}{h_{n}}%
\right) \right] \mathbb{I}_{_{A_{i}\cap B_{i}}}+ \\
& +\sum_{i=1}^{n}K\left( \frac{V_{i}}{h_{n}}\right) \mathbb{I}_{_{A_{i}\cap 
		\overline{B}_{i}}}-\sum_{i=1}^{n}K\left( \frac{\widehat{V}_{i}}{h_{n}}%
\right) \mathbb{I}_{_{_{\overline{A}_{i}\cap B_{i}}}}.
\end{align*}
Since $\left( a+b\right) ^{2}\leq 2a^{2}+2b^{2}$,
\begin{align}
\mathbb{E}\left[ S_{n}\left( x\right) -\widehat{S}_{n}\left( x\right) \right]
^{2}\leq & 2\mathbb{E}\left[ \sum_{i=1}^{n}\left( K\left( \frac{V_{i}}{h_{n}}%
\right) -K\left( \frac{\widehat{V}_{i}}{h_{n}}\right) \right) \mathbb{I}%
_{_{A_{i}\cap B_{i}}}\right] ^{2}  \notag \\
& +2\mathbb{E}\left[ \left( \sum_{i=1}^{n}K\left( \frac{V_{i}}{h_{n}}\right) 
\mathbb{I}_{_{A_{i}\cap \overline{B}_{i}}}\right) ^{2}+\left(
\sum_{i=1}^{n}K\left( \frac{\widehat{V}_{i}}{h_{n}}\right) \mathbb{I}_{_{_{%
			\overline{A}_{i}\cap B_{i}}}}\right) ^{2}\right] .
\label{eq:media2_maggioraz}
\end{align}
Consider now the first addend in the right--hand side of \eqref{eq:media2_maggioraz}: Assumption~\ref{ass:kernel} and the fact that $\left\vert V_{i}-\widehat{V}_{i}\right\vert \leq \left\Vert \Pi _{d}-\widehat{\Pi }_{d}\right\Vert _{\infty }\left\Vert X_{i}-x\right\Vert $, where $\left\Vert\cdot \right\Vert _{\infty }$\ denotes the operator norm, lead to
\begin{equation*}
\mathbb{E}\left[ \sum_{i=1}^{n}\left( K\left( \frac{V_{i}}{h_{n}}\right)
-K\left( \frac{\widehat{V}_{i}}{h_{n}}\right) \right) \mathbb{I}%
_{_{A_{i}\cap B_{i}}}\right] ^{2}\leq C\mathbb{E}\left[ \left\Vert \Pi _{d}-%
\widehat{\Pi }_{d}\right\Vert _{\infty }\sum_{i=1}^{n}\left\Vert
X_{i}-x\right\Vert \mathbb{I}_{_{A_{i}\cap B_{i}}}\right] ^{2}.
\end{equation*}
Thanks to the Cauchy-Schwartz inequality we control the previous bound by
\begin{equation}
C\mathbb{E}\left[ \left\Vert \Pi _{d}-\widehat{\Pi }_{d}\right\Vert _{\infty
}^{2}\right] \mathbb{E}\left[ \left( \sum_{i=1}^{n}\left\Vert
X_{i}-x\right\Vert \mathbb{I}_{_{A_{i}\cap B_{i}}}\right) ^{2}\right] .
\label{eq:CS-bound-proj}
\end{equation}%
About the first factor in \eqref{eq:CS-bound-proj}, \citealp[Theorem 2.1 (ii)]{bia:mas12} established that
\begin{equation}
\mathbb{E}\left[ \left\Vert \Pi _{d}-\widehat{\Pi }_{d}\right\Vert _{\infty
}^{2}\right] =O\left( \frac{1}{n}\right) .  \label{eq:bound_proj}
\end{equation}%
Consider now the second term in \eqref{eq:CS-bound-proj}. Thanks to the
Chebyshev's algebraic inequality (see, for instance, \citealp[page 243]{mit:pec:fin93}) and since $\mathbb{E}\left[ \mathbb{I}_{A_{i}\cap B_{i}}\right]
\leq \mathbb{E}\left[ \mathbb{I}_{A_{i}}\right] $, for any $k\geq 1$ it holds%
\begin{equation*}
\mathbb{E}\left[ \left\Vert X-x\right\Vert ^{k}\mathbb{I}_{A_{i}\cap B_{i}}%
\right] \leq \mathbb{E}\left[ \left\Vert X-x\right\Vert ^{k}\right] \mathbb{E%
}\left[ \mathbb{I}_{A_{i}}\right] .
\end{equation*}%
The fact that $\mathbb{E}\left[ \mathbb{I}_{A_{i}}\right] \sim h_{n}^{d}$
and Assumption~\ref{ass:moments_of_X} give
\begin{equation*}
\mathbb{E}\left[ \left\Vert X-x\right\Vert ^{k}\mathbb{I}_{A_{i}\cap B_{i}}%
\right] \leq C\frac{k!}{2}b^{k-2}h_{n}^{d},
\end{equation*}%
with $b>0$. Hence, the Bernstein inequality (see e.g.~\citealp{massart07}) can be
applied: for any $M>0$,
\begin{equation*}
\mathbb{P}\left( \left\vert \sum_{i=1}^{n}\left\Vert X_{i}-x\right\Vert 
\mathbb{I}_{_{A_{i}\cap B_{i}}}-\mathbb{E}\left[ \sum_{i=1}^{n}\left\Vert
X_{i}-x\right\Vert \mathbb{I}_{_{A_{i}\cap B_{i}}}\right] \right\vert \geq
Mnh^{d}\right) \leq \exp \left( -CM^{2}nh^{d}\right) .
\end{equation*}%
This result together with the Borel-Cantelli lemma lead to:%
\begin{equation*}
\sum_{i=1}^{n}\left\Vert X_{i}-x\right\Vert \mathbb{I}_{_{A_{i}\cap
		B_{i}}}\leq Cnh^{d}\ \ \ \ \ \ \ \ \ a.s.
\end{equation*}%
and therefore,%
\begin{equation}
\mathbb{E}\left[ \left( \sum_{i=1}^{n}\left\Vert X_{i}-x\right\Vert \mathbb{I%
}_{_{A_{i}\cap B_{i}}}\right) ^{2}\right] \leq Cn^{2}h^{2d}.
\label{eq:bound_sum2}
\end{equation}%
Finally, combining results \eqref{eq:bound_proj} and \eqref{eq:bound_sum2},
we obtain:%
\begin{equation}
\frac{1}{\left( nh_{n}^{d}\right) ^{2}}\mathbb{E}\left[ \sum_{i=1}^{n}\left(
K\left( \frac{V_{i}}{h_{n}}\right) -K\left( \frac{\widehat{V}_{i}}{h_{n}}%
\right) \right) \mathbb{I}_{_{A_{i}\cap B_{i}}}\right] ^{2}\leq C\frac{1}{%
	nh_{n}^{2}}.  \label{eq:media2_bound_1}
\end{equation}%
Consider now the second addend in the right--hand side of \eqref{eq:media2_maggioraz}. We only look at
\begin{equation}
\mathbb{E}\left[ \sum_{i=1}^{n}K\left( \frac{V_{i}}{h_{n}}\right) \mathbb{I}_{_{A_{i}\cap \overline{B}_{i}}}\right] ^{2},
\label{eq:media2_secondo_termine}
\end{equation}%
because the behaviour of the other addend is similar. Define the sequence $\kappa
_{n}$ so that $\kappa _{n}\rightarrow 0$ as $n\rightarrow \infty $, the
following inclusions hold: 
\begin{align*}
A_{i}\cap \overline{B}_{i}& =\left\{ V_{i}\leq h_{n}\right\} \cap \left\{ 
\widehat{V}_{i}>h_{n}\right\}  \\
& =\left( \left\{ h_{n}\left( 1-\kappa _{n}\right) <V_{i}\leq h_{n}\right\}
\cup \left\{ V_{i}\leq h_{n}\left( 1-\kappa _{n}\right) \right\} \right)
\cap \left\{ \widehat{V}_{i}-V_{i}>h_{n}-V_{i}\right\}  \\
& \subseteq \left\{ h_{n}\left( 1-\kappa _{n}\right) <V_{i}\leq
h_{n}\right\} \cup \left\{ V_{i}\leq h_{n}\left( 1-\kappa _{n}\right) ,%
\widehat{V}_{i}-V_{i}>h_{n}-V_{i}\right\}  \\
& \subseteq \left\{ h_{n}\left( 1-\kappa _{n}\right) <V_{i}\leq
h_{n}\right\} \cup \left\{ \widehat{V}_{i}-V_{i}>\kappa _{n}h_{n}\right\} .
\end{align*}%
The latter inclusion and Assumption~\ref{ass:kernel} allow to control \eqref{eq:media2_secondo_termine} by 
\begin{equation}
\mathbb{E}\left[ \sum_{i=1}^{n}\mathbb{I}_{A_{i}\cap \overline{B}_{i}}\right]
^{2}\leq 2\mathbb{E}\left[ \sum_{i=1}^{n}\mathbb{I}_{\left\{ h_{n}\left(
	1-\kappa _{n}\right) <V_{i}\leq h_{n}\right\} }\right] ^{2}+2\mathbb{E}\left[
\sum_{i=1}^{n}\mathbb{I}_{\left\{ \left\Vert \widehat{\Pi }_{d}-\Pi
	_{d}\right\Vert \left\Vert X_{i}-x\right\Vert >C\kappa _{n}h_{n}\right\} }%
\right] ^{2}.  \label{eq:media_somma_1}
\end{equation}%
About the first term in the right--hand side of the latter, the Cauchy-Schwartz inequality gives
\begin{equation*}
\mathbb{E}\left[ \sum_{i=1}^{n}\mathbb{I}_{\left\{ h_{n}\left( 1-\kappa
	_{n}\right) <V_{i}\leq h_{n}\right\} }\right] ^{2}\leq n^{2}\mathbb{P}\left(
h_{n}\left( 1-\kappa _{n}\right) <V\leq h_{n}\right) .
\end{equation*}%
Since $\mathbb{P}\left( h_{n}\left( 1-\kappa _{n}\right) <V\leq h_{n}\right)
\sim h_{n}^{d}\left( 1-\left( 1-\kappa _{n}\right) ^{d}\right) $, performing
a first order Taylor expansion of $\left( 1-\kappa _{n}\right) ^{d}$ in $%
\kappa _{n}=0$, we get asymptotically
\begin{equation*}
\mathbb{E}\left[ \sum_{i=1}^{n}\mathbb{I}_{\left\{ h_{n}\left( 1-\kappa
	_{n}\right) <V_{i}\leq h_{n}\right\} }\right] ^{2}\leq Cn^{2}h_{n}^{d}\kappa
_{n}.
\end{equation*}%
Similarly, for what concerns the other addend in the right--hand side of %
\eqref{eq:media_somma_1}, we have%
\begin{equation*}
\mathbb{E}\left[ \sum_{i=1}^{n}\mathbb{I}_{\left\{ \left\Vert \widehat{\Pi }%
	_{d}-\Pi _{d}\right\Vert \left\Vert X_{i}-x\right\Vert >C\kappa
	_{n}h_{n}\right\} }\right] ^{2}\leq n^{2}\mathbb{P}\left( \left\Vert 
\widehat{\Pi }_{d}-\Pi _{d}\right\Vert \left\Vert X-x\right\Vert >C\kappa
_{n}h_{n}\right) .
\end{equation*}%
Thanks to the Markov inequality, \citealp[Theorem 2.1 (iii)]{bia:mas12} and
Assumption~\ref{ass:moments_of_X}, it follows%
\begin{equation*}
\mathbb{P}\left( \left\Vert \widehat{\Pi }_{d}-\Pi _{d}\right\Vert
\left\Vert X-x\right\Vert >C\kappa _{n}h_{n}\right) =O\left( \frac{1}{%
	n^{1/2}h_{n}\kappa _{n}}\right) .
\end{equation*}%
Combining the previous results we obtain:%
\begin{equation*}
\frac{1}{\left( nh_{n}^{d}\right) ^{2}}\mathbb{E}\left[ \left(
\sum_{i=1}^{n}K\left( \frac{V_{i}}{h_{n}}\right) \mathbb{I}_{_{A_{i}\cap 
		\overline{B}_{i}}}\right) \right] ^{2}=O\left( \frac{\kappa _{n}}{h_{n}^{d}}%
\right) +O\left( \frac{1}{n^{1/2}h_{n}\kappa _{n}}\right) .
\end{equation*}%
If we choose $\kappa _{n} = \left( n^{5/2}h_{n}^{2d}\right) ^{-1/2}\ $with$\ \
n^{5/4}h_{n}^{d}\rightarrow \infty $, as $n\rightarrow \infty $, we obtain: 
\begin{equation}
\mathbb{E}\left[ \left( \sum_{i=1}^{n}K\left( \frac{V_{i}}{h_{n}}\right) 
\mathbb{I}_{_{A_{i}\cap \overline{B}_{i}}}\right) ^{2}+\left(
\sum_{i=1}^{n}K\left( \frac{\widehat{V}_{i}}{h_{n}}\right) \mathbb{I}_{_{_{%
			\overline{A}_{i}\cap B_{i}}}}\right) ^{2}\right] \leq C\frac{1}{%
	n^{5/4}h_{n}^{2d}}.  \label{eq:media2_bound_2}
\end{equation}%
In conclusion, \eqref{eq:media2_bound_1} and \eqref{eq:media2_bound_2} lead
to:%
\begin{equation*}
\frac{1}{\left( nh_{n}^{d}\right) ^{2}}\mathbb{E}\left[ S_{n}\left( x\right)
-\widehat{S}_{n}\left( x\right) \right] ^{2}=O\left( \frac{1}{nh_{n}^{2}}%
\right) +O\left( \frac{1}{n^{5/4}h_{n}^{2d}}\right) .
\end{equation*}%
Choose the optimal bandwidth \eqref{eq:optimal_bandwidth} and $p>2\vee 3d/10$, then, as $n$ goes to infinity, the first addend becomes negligible compared to the second one that turns to be $O\left( n^{-\left(10p-3d\right) /4\left( 2p+d\right) }\right) $. Moreover, a direct computation shows that such bound is definitively negligible when compared to the ``optimal bound'' $n^{-2p/\left( 2p+d\right) }$, for any $p>2\vee 3d/2$ and $d\geq 1$. This concludes the proof.

\paragraph{Acknowledgements}

%
The authors are grateful to P.~Vieu for discussions held while visiting University Paul Sabatier in Toulouse whose hospitality is appreciated. E.~Bongiorno thanks G.~Gonz\'{a}lez--Rodr\'{\i}guez for the fruitful discussions they had during a visit of the author at Universidad de Oviedo (towards which he is grateful for the kind hospitality).
The authors are members of the Gruppo Nazionale per l'Analisi Matematica, la Probabilit\`{a} e le loro Applicazioni (GNAMPA) of the Istituto Nazionale di Alta Matematica (INdAM) that partially funded this work.

\DeclareRobustCommand{\VAN}[3]{#2}
\bibliographystyle{plain}
\bibliography{SmBProb_biblio}

\begin{thebibliography}{24}
\expandafter\ifx\csname natexlab\endcsname\relax\def\natexlab#1{#1}\fi
\expandafter\ifx\csname url\endcsname\relax
  \def\url#1{\texttt{#1}}\fi
\expandafter\ifx\csname urlprefix\endcsname\relax\def\urlprefix{URL }\fi

\bibitem[{Berlinet and Thomas-Agnan(2004)}]{ber:tho04}
Berlinet, A., Thomas-Agnan, C., 2004. Reproducing kernel {H}ilbert spaces in
  probability and statistics. Kluwer Academic Publishers, Boston, MA.

\bibitem[{Biau and Mas(2012)}]{bia:mas12}
Biau, G., Mas, A., 2012. P{CA}-kernel estimation. Stat. Risk Model. 29~(1),
  19--46.

\bibitem[{Bongiorno and Goia(2016)}]{bon:goi2016}
Bongiorno, E.~G., Goia, A., 2016. Classification methods for hilbert data based
  on surrogate density. Comput. Statist. Data Anal. 99, 204 -- 222.

\bibitem[{Bongiorno et~al.(2014)Bongiorno, Goia, Salinelli, and
  Vieu}]{bon:goi:sal:vie14}
Bongiorno, E.~G., Goia, A., Salinelli, E., Vieu, P. (Eds.), 2014. Contributions
  in infinite-dimensional statistics and related topics. Societ\`{a} Editrice
  Esculapio.

\bibitem[{Bosq(2000)}]{bos00}
Bosq, D., 2000. Linear processes in function spaces. Vol. 149 of Lecture Notes
  in Statistics. Springer-Verlag, New York.

\bibitem[{Box and Tiao(1973)}]{box:tia73}
Box, G. E.~P., Tiao, G.~C., 1973. Bayesian inference in statistical analysis.
  Addison-Wesley Publishing Co., Reading, Mass.-London-Don Mills, Ont.

\bibitem[{Dabo-Niang et~al.(2007)Dabo-Niang, Ferraty, and Vieu}]{dab:fer:vie07}
Dabo-Niang, S., Ferraty, F., Vieu, P., 2007. On the using of modal curves for
  radar waveforms classification. Comput. Statist. Data Anal. 51~(10),
  4878--4890.

\bibitem[{Delaigle and Hall(2010)}]{del:hal10}
Delaigle, A., Hall, P., 2010. Defining probability density for a distribution
  of random functions. Ann. Statist. 38~(2), 1171--1193.

\bibitem[{Dereich et~al.(2003)Dereich, Fehringer, Matoussi, and
  Scheutzow}]{der:feh:mat:sch03}
Dereich, S., Fehringer, F., Matoussi, A., Scheutzow, M., 2003. On the link
  between small ball probabilities and the quantization problem for {G}aussian
  measures on {B}anach spaces. J. Theoret. Probab. 16~(1), 249--265.

\bibitem[{Duong(2007)}]{duo07}
Duong, T., 10 2007. ks: Kernel density estimation and kernel discriminant
  analysis for multivariate data in r. J. Stat. Softw. 21~(7), 1--16.

\bibitem[{Ferraty et~al.(2012)Ferraty, Kudraszow, and Vieu}]{fer:kud:vie12}
Ferraty, F., Kudraszow, N., Vieu, P., 2012. Nonparametric estimation of a
  surrogate density function in infinite-dimensional spaces. J. Nonparametr.
  Stat. 24~(2), 447--464.

\bibitem[{Ferraty et~al.(2007)Ferraty, Mas, and Vieu}]{fer:mas:vie07}
Ferraty, F., Mas, A., Vieu, P., 2007. Nonparametric regression on functional
  data: inference and practical aspects. Aust. N. Z. J. Stat. 49~(3), 267--286.

\bibitem[{Ferraty and Vieu(2006)}]{fer:vie06}
Ferraty, F., Vieu, P., 2006. Nonparametric functional data analysis. Springer
  Series in Statistics. Springer, New York.

\bibitem[{Gasser et~al.(1998)Gasser, Hall, and Presnell}]{gas:hal:pre98}
Gasser, T., Hall, P., Presnell, B., 1998. Nonparametric estimation of the mode
  of a distribution of random curves. J. R. Stat. Soc. Ser. B Stat. Methodol.
  60~(4), 681--691.

\bibitem[{Hall et~al.(2006)Hall, M{\"u}ller, and Wang}]{hal:mul:wan06}
Hall, P., M{\"u}ller, H.-G., Wang, J.-L., 2006. Properties of principal
  component methods for functional and longitudinal data analysis. Ann.
  Statist. 34~(3), 1493--1517.

\bibitem[{Horv{\'a}th and Kokoszka(2012)}]{hor:kok12}
Horv{\'a}th, L., Kokoszka, P., 2012. Inference for functional data with
  applications. Vol. 200. Springer Science \& Business Media.

\bibitem[{Li and Shao(2001)}]{li:sha01}
Li, W.~V., Shao, Q.-M., 2001. Gaussian processes: inequalities, small ball
  probabilities and applications. In: Stochastic processes: theory and methods.
  Vol.~19 of Handbook of Statist. North-Holland, Amsterdam, pp. 533--597.

\bibitem[{Lifshits(2012)}]{lif12}
Lifshits, M.~A., 2012. Lectures on {G}aussian processes. Springer Briefs in
  Mathematics. Springer, Heidelberg.

\bibitem[{Massart(2007)}]{massart07}
Massart, P., 2007. Concentration inequalities and model selection. Vol. 1896 of
  Lecture Notes in Mathematics. Springer, Berlin, lectures from the 33rd Summer
  School on Probability Theory held in Saint-Flour, July 6--23, 2003, With a
  foreword by Jean Picard.

\bibitem[{Mitrinovi{\'c} et~al.(1993)Mitrinovi{\'c}, Pe{\v{c}}ari{\'c}, and
  Fink}]{mit:pec:fin93}
Mitrinovi{\'c}, D.~S., Pe{\v{c}}ari{\'c}, J.~E., Fink, A.~M., 1993. Classical
  and new inequalities in analysis. Vol.~61 of Mathematics and its Applications
  (East European Series). Kluwer Academic Publishers Group, Dordrecht.

\bibitem[{Ramsay and Silverman(2005)}]{ram:sil05}
Ramsay, J.~O., Silverman, B.~W., 2005. Functional data analysis, 2nd Edition.
  Springer Series in Statistics. Springer, New York.

\bibitem[{Rice and Silverman(1991)}]{ric:sil91}
Rice, J.~A., Silverman, B.~W., 1991. Estimating the mean and covariance
  structure nonparametrically when the data are curves. J. Roy. Statist. Soc.
  Ser. B 53~(1), 233--243.

\bibitem[{Shiryayev(1984)}]{shi95}
Shiryayev, A.~N., 1984. Probability. Vol.~95 of Graduate Texts in Mathematics.
  Springer-Verlag, New York.

\bibitem[{Wand and Jones(1995)}]{wan:jon95}
Wand, M.~P., Jones, M.~C., 1995. Kernel smoothing. Vol.~60 of Monographs on
  Statistics and Applied Probability. Chapman and Hall, Ltd., London.

\end{thebibliography}

\end{document}